\documentclass[11pt]{amsart}
\usepackage[english]{babel}
\usepackage{amssymb}
\usepackage{enumerate}
\usepackage{amsthm}
\usepackage{amsmath} 
\usepackage[all,cmtip]{xy}

\usepackage{tikz}
\usetikzlibrary{decorations.pathreplacing}

\usepackage[margin=1in]{geometry}
\usepackage{setspace}
\usepackage{hyperref}

\newtheorem{theorem}{Theorem}
\newtheorem{corollary}[theorem]{Corollary}
\newtheorem{prop}[theorem]{Proposition}
\newtheorem{lemma}[theorem]{Lemma}

\theoremstyle{definition} 
\newtheorem{remark}[theorem]{Remark}
\theoremstyle{remark} 

\numberwithin{theorem}{section}

\newcommand{\Ric}{Ric}
\newcommand{\Rm}{Rm} 
\DeclareMathOperator{\area}{area}
\DeclareMathOperator{\vol}{vol}
\newcommand{\R}{\mathbb{R}}
\DeclareMathOperator{\dist}{dist}
\DeclareMathOperator{\proj}{proj}

\newcommand{\scal}{R}

\usepackage[normalem]{ulem}

\title{Effective versions of the positive mass theorem}
\author{Alessandro Carlotto}
\address{ETH Institute for Theoretical Studies, Clausiusstrasse 47, 8092 Zurich, Switzerland}
\email{alessandro.carlotto@eth-its.ethz.ch}
\author{Otis Chodosh}
\address{DPMMS, University of Cambridge, Wilberforce Road, Cambridge, CB3 0WB, United Kingdom}
\email{oc249@cam.ac.uk}
\author{Michael Eichmair}
\address{Faculty of Mathematics, University of Vienna, Oskar-Morgenstern-Platz 1, 1090 Vienna, Austria}
\email {michael.eichmair@univie.ac.at}

\begin{document}

\maketitle

\begin{abstract} The study of stable minimal surfaces in Riemannian $3$-manifolds $(M, g)$ with non-negative scalar curvature has a rich history. In this paper, we prove rigidity of such surfaces when $(M, g)$ is asymptotically flat and has horizon boundary. As a consequence, we obtain an effective version of the positive mass theorem in terms of isoperimetric or, more generally, closed volume-preserving stable CMC surfaces that is appealing from both a physical and a purely geometric point of view. We also include a proof of the following conjecture of R. Schoen: An asymptotically flat Riemannian $3$-manifold with non-negative scalar curvature that contains an unbounded area-minimizing surface is isometric to flat $\R^3$.
\end{abstract}


\section {Introduction}

The geometry of stable minimal and volume-preserving stable constant mean curvature surfaces in asymptotically flat $3$-manifolds $(M, g)$ with non-negative scalar curvature is witness to the physical properties of the space-times containing such $(M, g)$ as maximal Cauchy hypersurfaces; see e.g. \cite{Penrose:1965, Schoen-Yau:PMT1, Christodoulou-Yau:1988, Huisken-Yau:1996, Bray:1997, Bray:2001, Huisken-Ilmanen:2001}. The transition from positive to non-negative scalar curvature of $(M, g)$, which is so crucial for physical applications, is a particularly delicate aspect in the analysis of such surfaces. Here we establish optimal rigidity results in this context that apply very generally. We apply them to obtain a precise understanding of the behavior of large isoperimetric or, more generally, closed volume-preserving stable constant mean curvature surfaces in $(M, g)$ that extends the results of S. Brendle, J. Metzger, and the third-named author \cite{stableCMC, isostructure, hdiso, offcenter}. In combination with existing literature, this yields a rather complete analogy between the picture in $(M, g)$ and classical results in Euclidean space.   

We review the standard terminology and conventions that we use here in Appendix \ref{sec:definitions}. In particular, we follow the convention that stable minimal surfaces are by definition two-sided.  

To provide context, we recall a celebrated application of the second variation of area formula due to R. Schoen and S.-T. Yau \cite[Theorem 6.1]{Schoen-Yau:1978}. Assume (for contradiction) that we are given a metric of positive scalar curvature on the $3$-torus $\mathbb{T}^3$. Using results from geometric measure theory, one can find a closed surface $\Sigma \subset \mathbb{T}^3$ of non-zero genus that minimizes area in its homology class with respect to this metric. In particular, $\Sigma$ is a stable minimal surface. Using the function $u = 1$ in the stability inequality \eqref{eqn:stabilityminimal}, we obtain that 
\[
0 \geq  \int_\Sigma |h|^2 + \Ric (\nu, \nu).
\] 
We may rewrite the integrand as 
\[
|h|^2 + \Ric (\nu, \nu) = \frac{1}{2} ( |h|^2 + R) - K, 
\]
using the Gauss equation \eqref{eqn:Gaussequation}. It follows that
\[
\int_\Sigma K > 0
\]
which is incompatible with the Gauss-Bonnet formula. Thus $\mathbb{T}^3$ does \emph{not} admit a metric of positive scalar curvature.

This crucial mechanism --- \emph{positive ambient scalar curvature is incompatible with the existence of stable minimal surfaces of most topological types} --- is at the heart of another fundamental result proven by R. Schoen and S.-T. Yau, the positive mass theorem \cite{Schoen-Yau:PMT1}: If $(M, g)$ is asymptotically flat with horizon boundary and non-negative integrable scalar curvature, then its ADM-mass is non-negative. Moreover, the ADM-mass vanishes if and only if $(M, g)$ is isometric to Euclidean space. Using an initial perturbation, they reduce the proof of non-negativity of the ADM-mass to the special case where $(M, g)$ is asymptotic to Schwarzschild with horizon boundary and positive scalar curvature. If the mass is negative, then the coordinate planes $\{x^3 = \pm \Lambda\}$ with respect to the chart at infinity act as barriers for area minimization in the slab-like region they enclose in $M$, provided $\Lambda > 1$ is sufficiently large. Using geometric measure theory, one finds an unbounded complete area-minimizing boundary $\Sigma$ in this slab. Such a surface has quadratic area growth. Using the logarithmic cut-off trick in the second variation of area (observing the decay of the ambient Ricci curvature to handle integrability issues), it follows as before that 
\[
0 < \frac{1}{2} \int_\Sigma |h|^2 + R = \int_\Sigma K.
\] 
A result of S. Cohn-Vossen shows that $\Sigma \cong \R^2$. Using that $\Sigma$ is area-minimizing in a slab, they argue that $\Sigma$ is asymptotic to a horizontal plane and conclude that the geodesic curvature of the circles $\Sigma \cap S_r$ in $\Sigma$ converges to $2 \pi$ as $r \to \infty$.\footnote{An alternative argument for this part of the proof that also generalizes to stable minimal surfaces with quadratic area growth was given in \cite[Proposition 3.6]{stableCMC}. The strategy of \cite{stableCMC} is exploited in the proof of Theorem \ref{thm:stableminimalproper} below.} The Gauss-Bonnet formula gives that 
\[
\int_\Sigma K = 0,  
\] 
a contradiction. It follows that the ADM-mass of $(M, g)$ is non-negative.  

Observe that this line of reasoning \emph{cannot} establish the rigidity part (only Euclidean space has vanishing mass) of the positive mass theorem. Conversely, a beautiful idea of J. Lohkamp \cite[Section 6]{Lohkamp:1999} shows that the rigidity assertion of the positive mass theorem implies the non-negativity of ADM-mass in general. Indeed, he shows that it suffices to prove that an asymptotically flat Riemannian $3$-manifold with horizon boundary and non-negative scalar curvature is flat if it is flat outside of a compact set.

The ideas of R. Schoen and S.-T. Yau described above are instrumental to our results here. We record the following technical variation on their work as a precursor of Theorems \ref{thm:Carlotto} and \ref{thm:stableminimalproper} below.

\begin {prop} [Section 6 in \cite{isostructure}] \label{prop:simplepmt} Let $(M, g)$ be an asymptotically flat Riemannian $3$-manifold. Assume that $\Sigma \subset M$ is the unbounded component of an area-minimizing boundary in $(M, g)$, and that the scalar curvature of $(M, g)$ is non-negative along $\Sigma$. Then $\Sigma \subset M$ is totally geodesic and the scalar curvature of $(M, g)$ vanishes along this surface. Moreover, for all $\rho > 1$ sufficiently large, $\Sigma$ intersects $S_\rho$ transversely in a nearly equatorial circle. The Gauss curvature is integrable and $\int_\Sigma K = 0$.
\end {prop}

We also mention that other proofs of the positive mass theorem (including that of E. Witten \cite{Witten:1981} based on the Bochner formula for harmonic spinors and that of G. Huisken and T. Ilmanen \cite{Huisken-Ilmanen:2001} based on inverse mean curvature flow) have been given. 

The discoveries of R. Schoen and S.-T. Yau have incited a remarkable surge of activity investigating the relationship between scalar curvature, locally area-minimizing (or stable minimal) surfaces, and the physical properties of spacetimes evolving from asymptotically flat Riemannian $3$-manifolds according to the Einstein equations. This has lead to spectacular developments in geometry and physics. We refer the reader to \cite{Fischer-Colbrie-Schoen, Schoen:1984, Fischer-Colbrie:1985, Anderson-Rodriguez:1989, Corvino:2000, Bray:2001, Huisken-Ilmanen:2001} to gain an impression of the wealth and breadth of the repercussions. 

The following rigidity result for scalar curvature was first proven by the first-named author under the additional assumption of quadratic area growth for the surface $\Sigma$. Subsequently, the quadratic area growth assumption was removed independently (in the form of Theorem \ref{thm:Carlotto} below) by the first-named author \cite{Carlotto} and (in the form of Theorem \ref{thm:stableminimalproper} below) in a joint project of the second- and third-named authors. The proof of Theorem \ref{thm:stableminimalproper} is included in this paper. 

\begin {theorem}[\cite{Carlotto}] \label{thm:Carlotto} Let $(M, g)$ be an asymptotically flat Riemannian $3$-manifold with non-negative scalar curvature. Let $\Sigma \subset M$ be a non-compact properly embedded stable minimal surface. Then $\Sigma$ is a totally geodesic flat plane and the ambient scalar curvature vanishes along $\Sigma$. Such a surface cannot exist under the additional assumption that $(M, g)$ is asymptotic to Schwarzschild with mass $m >0$. 
\end {theorem}

\begin {theorem} \label{thm:stableminimalproper} Let $(M, g)$ be a Riemannian $3$-manifold with non-negative scalar curvature that is asymptotic to Schwarzschild with mass $m > 0$ and which has horizon boundary. Every complete stable minimal immersion $\varphi : \Sigma \to M$ that is proper is an embedding of a component of the horizon. 
\end {theorem}

To obtain these results, it is necessary to understand how non-negative scalar curvature keeps in check the \emph{a priori} wild behavior at infinity of the minimal surface. This difficulty does not arise in the original argument by R. Schoen and S.-T. Yau. The proofs of Theorems \ref{thm:Carlotto} and \ref{thm:stableminimalproper} use properness in a crucial way. Moreover, the embeddedness assumption is essential in the derivation of Theorem \ref{thm:Carlotto} in \cite{Carlotto}. 

In spite of their geometric appeal, we cannot apply Theorems \ref{thm:Carlotto} and \ref{thm:stableminimalproper} to prove effective versions of the positive mass theorem such as Theorem \ref{thm:largeCMC} below. This is intimately related to the fact that properness is not preserved by convergence of immersions. Our first main contribution here is the following technical result that rectifies this: 

\begin {theorem} \label{thm:reduction} Let $(M, g)$ be an asymptotically flat Riemannian $3$-manifold with non-negative scalar curvature. Assume that there is an unbounded complete stable minimal immersion $\varphi : \Sigma \to M$ that does not cross itself. Then $(M, g)$ admits a complete non-compact properly embedded stable minimal surface. 
\end {theorem}

Using this, we obtain the following substantial improvement of Theorems \ref{thm:Carlotto} and \ref{thm:stableminimalproper}:

\begin {theorem} \label{thm:stableminimalnocrossing} Let $(M, g)$ be a Riemannian $3$-manifold with non-negative scalar curvature that is asymptotic to Schwarzschild with mass $m > 0$ and which has horizon boundary. The only non-trivial complete stable minimal immersions $\varphi : \Sigma \to M$ that do not cross themselves are embeddings of components of the horizon.
\end {theorem}


For the proof of Theorem \ref{thm:reduction}, we develop in Section \ref{sec:topsheets} a general procedure of extracting \emph{properly embedded top sheets} from unbounded complete stable minimal immersions that do not cross themselves. The method depends on a purely analytic stability result --- Corollary \ref{cor:nocylinders} --- that restricts the topological type of complete stable minimal immersions into $(M, g)$. 

The proof of the positive mass theorem suggests the following conjecture \cite[p. 48]{Schoen:talk} of R. Schoen: \emph{An asymptotically flat Riemannian manifold with non-negative scalar curvature that contains an unbounded area-minimizing surface is isometric to Euclidean space.} We include here a proof of this conjecture and that of a related rigidity result for slabs, both due to the second- and third-named authors: 

\begin {theorem} \label{thm:Schoenrigidity} The only asymptotically flat Riemannian $3$-manifold with non-negative scalar curvature that admits a non-compact area-minimizing boundary is flat $\R^3$. 
\end {theorem}

We recall the precise meaning of \emph{area-minimizing boundaries} in Appendix \ref{sec:am}. 

\begin {theorem} \label{thm:rigidityslabs} Let $(M, g)$ be an asymptotically flat Riemannian $3$-manifold with non-negative scalar curvature and with horizon boundary. Any two disjoint connected unbounded properly embedded complete minimal surfaces in $(M, g)$ bound a region that is isometric to a standard Euclidean slab $\R^2 \times [a, b]$. 
\end {theorem}

The proofs of Theorem \ref{thm:Schoenrigidity} and \ref{thm:rigidityslabs} are inspired by the recent refinement due to G. Liu \cite{Liu:2013} of a strategy of M. Anderson and L. Rodr\'iguez \cite{Anderson-Rodriguez:1989} to prove rigidity results for complete manifolds with non-negative \emph{Ricci} curvature.

We point that that we may excise the slab in the conclusion of Theorem \ref{thm:rigidityslabs} from $(M,g)$ to produce a new \emph{smooth} asymptotically flat Riemannian $3$-manifold with non-negative scalar curvature that contains a properly embedded totally geodesic flat plane along which the ambient scalar curvature vanishes. 

For comparison, we recall the following consequence of a gluing result due to the first-named author and R. Schoen: 

\begin {theorem} [\cite{Carlotto-Schoen}] \label{thm:CS} There exists an asymptotically flat Riemannian metric $g = g_{ij} \, dx^i \otimes dx^j$ with non-negative scalar curvature and positive mass on $\R^3$ such that $g_{ij} = \delta_{ij}$ on $\R^2 \times (0, \infty)$.
\end {theorem} 

The coordinate planes $\R^2 \times \{z\}$ with $z > 0$ in Theorem \ref{thm:CS} are stable minimal surfaces. In particular, the area-minimizing condition in Theorem \ref{thm:Schoenrigidity} cannot be relaxed. We also see that the condition that $(M, g)$ be asymptotic to Schwarzschild in Theorem \ref{thm:stableminimalnocrossing} is necessary. 

There is a rich theory of rigidity results for (compact) minimal surfaces in Riemannian $3$-manifolds with a lower scalar-curvature bound. We refer the reader to the papers \cite{Cai-Galloway:2000,Bray-Brendle-Eichmair-Neves:2010,Bray-Brendle-Neves:2010,MarquesNeves:min-max-rigidity-3mflds,MaximoNunes:hawking-rigidity,Nunes:hyperbolic-rigidity,Ambrozio:free-bdry-rigidity,MicallefMoraru} for several recent results in this direction, and to the introductions of these papers for a complete overview.

Theorem \ref{thm:Schoenrigidity} plays a role in the classification of initial data sets that admit a global static potential. Let $(M, g)$ be a connected Riemannian manifold that admits a non-constant function $f \in C^\infty(M)$ with $L^* f = 0$, where
\[
 L^* f = - (\Delta f) g +  \nabla d f  - f \,  \text{Ric} 
\]
is the formal adjoint of the linearisation of the scalar curvature operator at $g$. We recall from e.g. \cite{Corvino:2000} that when $(M, g)$ is asymptotically flat, then its scalar curvature vanishes and the condition that $L^* f = 0$ is equivalent to
\[
\nabla d f = f \, \text{Ric}  \quad \text{ and } \quad \Delta f = 0,
\]
implying that the spacetime 
\[
(M^o \times \R, g - f^2 d t \otimes d t ) \quad \text{ where }  \quad M^o = \{x \in M : f(x) > 0\}
\]
is a static solution of the vacuum Einstein equations. More generally, G. Galloway and P. Miao show in \cite{Galloway-Miao:2014} that when $(M, g)$ is an asymptotically flat Riemannian $3$-manifold --- possibly with several ends --- such that $f$ vanishes on the boundary of $M$, then every unbounded component of the (necessarily regular) level set $\{x \in M : f(x) = 0\}$ is an absolutely area-minimizing plane. As observed in Section 4 of \cite{Galloway-Miao:2014}, Theorem \ref{thm:Schoenrigidity} shows that such unbounded components can only exist when $(M, g)$ is flat $\R^3$  and $f$ is a linear function. Together with Corollary 1.1 of \cite{Miao-Tam:2015} and the refinement of the results of  G. Bunting and A. K. M. Masood-ul-Alam \cite{Bunting-Masood-ul-Alam:1987} in Proposition 4.1 of \cite{Miao-Tam:2015}, both due to P. Miao and L.-T. Tam, one obtains the following classification result:

\begin{corollary} Let $(M, g)$ be an asymptotically flat Riemannian $3$-manifold, possibly with several ends, that admits a non-constant function $f \in C^\infty(M)$ with $L^* f = 0$ that vanishes on the boundary of $M$. Then $(M, g)$ is isometric to either flat $\R^3$, or, for some $m > 0$, either Schwarzschild 
\[
\left( \{x \in \R^3 : |x| \geq m/2 \}, \left(1 + \frac{m}{2 |x|}\right)^4 \sum_{i=1}^3 dx^i \otimes dx^i \right)
\]
or the doubled spatial Schwarzschild geometry
\[
\left( \R^3 \setminus \{0\}, \left(1 + \frac{m}{2 |x|}\right)^4 \sum_{i=1}^3 dx^i \otimes dx^i \right).
\]  
\end {corollary}
We are grateful to L. Ambrozio and P. Miao for valuable discussions concerning this point.

We now turn our attention to the role played by closed volume-preserving CMC surfaces in asymptotically flat manifolds. 

In their groundbreaking paper \cite{Huisken-Yau:1996}, G. Huisken and S.-T. Yau have shown that the complement of a certain (large) compact subset $C$ of a Riemannian $3$-manifold $(M, g)$ that is asymptotic to Schwarzschild with mass $m > 0$ admits a foliation by closed volume-preserving CMC spheres $\{\Sigma_H\}_{H \in (0, H_0]}$ where $\Sigma_H$ has (outward) mean curvature $H$. Importantly, they observed that each leaf $\Sigma_H$ is characterized uniquely by its mean curvature among a large class of surfaces, justifying reference to $\{\Sigma_H\}_{H \in (0, H_0]}$ as the canonical foliation of the end of $(M, g)$. In \cite{Qing-Tian:2007}, J. Qing and G. Tian have given a delicate improvement of this characterization showing that $\Sigma_H$ is in fact the only closed volume-preserving stable CMC sphere of mean curvature $H$ in $(M, g)$ that encloses $C$. These results of \cite{Huisken-Yau:1996, Qing-Tian:2007} are perturbative in nature in that only surfaces far out in the chart at infinity are considered. They lie very deep even in the special case of the exact  Schwarzschild (spatial) geometry 
\begin{align} \label{eqn:Schwarzschild}
\left( \mathbb{R}^3 \setminus \overline {B_{\frac{m}{2}} (0) }, \left( 1 + \frac{m}{2 |x|}\right)^4 \sum_{i=1}^3 dx^i \otimes dx^i\right).
\end{align}
We mention the spectacular recent characterization \cite{Brendle:2013} by S. Brendle of closed embedded constant mean curvature surfaces in Schwarzschild as the centered coordinate spheres in this context. 

In the next two main results, we investigate the question of global uniqueness results for large volume-preserving stable CMC surfaces in asymptotically flat manifolds.

\begin{theorem} \label{thm:largeCMC} Let $(M, g)$ be an asymptotically flat Riemannian $3$-manifold with non-negative scalar curvature and horizon boundary. Assume that $(M, g)$ contains no properly embedded totally geodesic flat planes along which the ambient scalar curvature vanishes. Let $C \subset M$ be compact. There is $\alpha = \alpha(C) > 0$ so that every connected closed volume-preserving stable CMC surface $\Sigma \subset M$ with 
\[
\area (\Sigma) \geq \alpha
\] 
is disjoint from $C$.
\end{theorem}

In conjunction with the uniqueness results from \cite{Huisken-Yau:1996, Qing-Tian:2007}, we obtain the following consequence:

\begin {corollary}  \label{cor:largeCMC} Let $(M, g)$ be a Riemannian $3$-manifold with non-negative scalar curvature that is asymptotic to Schwarzschild with mass $m > 0$ and which has horizon boundary. Let $p \in M$. Every connected closed volume-preserving stable CMC surface $\Sigma \subset M$ that contains $p$ and which has sufficiently large area is part of the canonical foliation. 
\end{corollary}

Theorem \ref{thm:largeCMC} was proven by the third-named author and J. Metzger in \cite{stableCMC} under the (much) stronger assumption that $(M, g)$ has \emph{positive} scalar curvature. As we have already mentioned, our improvement here is closely tied to the generality of Theorem \ref{thm:reduction}. 

In \cite{offcenter}, S. Brendle and the third-named author have constructed examples of Riemannian $3$-manifolds asymptotic to Schwarzschild with positive mass that contain a sequence of larger and larger volume-preserving stable CMC surfaces that diverge to infinity \emph{together} with the regions they bound. Thus, in the uniqueness results of \cite{Huisken-Yau:1996, Qing-Tian:2007}, a proviso that the surfaces enclose \emph{some} given set is certainly necessary. When the assumption of Schwarzschild asymptotics is dropped, the examples in Theorem \ref{thm:CS} show even more dramatically that some such a condition is necessary to obtain uniqueness results. Theorem \ref{thm:largeCMC} extends the results of \cite{Huisken-Yau:1996, Qing-Tian:2007} optimally in this sense. 

We remark that much progress has been made recently in developing analogues of the results of \cite{Huisken-Yau:1996, Qing-Tian:2007} in general asymptotically flat Riemannian $3$-manifolds, see e.g. \cite{Huang:2010, Ma:2011, Nerz:2014}.

D. Christodoulou and S.-T. Yau \cite{Christodoulou-Yau:1988} have noted that the Hawking mass of volume-preserving stable CMC spheres in asymptotically flat Riemannian $3$-manifolds with non-negative scalar curvature is non-negative. This observation makes these surfaces particularly appealing from a physical standpoint. Geometrically, they arise in the variational analysis of the fundamental question of isoperimetry. The results described above beg the question whether each leaf of the canonical foliation $\{\Sigma_H\}_{H \in (0, H_0]}$ has least area for the volume it encloses, and whether it is uniquely characterized by this property. This global uniqueness result was established by J. Metzger and the third-named author in \cite{isostructure}. (In exact Schwarzschild, this was proven by H. Bray in his dissertation \cite{Bray:1997}.) Unlike the results based on stability that we have described above, the existence and global uniqueness of isoperimetric regions of large volume has been verified in higher dimensions as well \cite{hdiso}. 

The definition of the ADM-mass through flux integrals as in \eqref{def:mass} and that of related \emph{physical} invariants that  canonically associated with an asymptotically flat Riemannian $3$-manifold $(M, g)$ is suggested by the Hamiltonian formalism of general relativity. The fact that the positive mass theorem was a longstanding open problem is witness to the elusive nature of these concepts. Over the past two decades, in a quest for quasi-local versions of these notions, considerable effort has been spent on recasting these concepts in terms of geometric properties of $(M, g)$. A spectacular advance in this direction is the development of an isoperimetric notion of mass by G. Huisken. Recall the classical fact that a small geodesic ball in a Riemannian manifold that is centered at a point of positive scalar curvature bounds more volume than a Euclidean ball of the same surface area. An explicit computation gives that large centered coordinate balls in Schwarzschild (which is scalar-flat) have the same property, and that the ``isoperimetric deficit" encodes the mass. G. Huisken has introduced the concept of \emph{isoperimetric mass}

\[
m_{ISO} = \lim_{r \to \infty} \frac{2}{\area (S_r) } \left( \vol (B_r) - \frac{\area (S_r)^{3/2}}{ 6 \sqrt \pi}  \right)
\] 
which does not involve derivatives of the metric at all. In \cite{Fan-Shi-Tam:2009}, X.-Q. Fan, P. Miao, Y. Shi, and L.-F. Tam have shown that 
\[
m_{ISO} = m_{ADM}
\]
if the scalar curvature of $(M, g)$ is integrable. Their derivation is based on a striking integration by parts. Thus, if $m_{ADM} > 0$, then large coordinate balls $B_r$ in $(M, g)$ contain more volume than balls of the same surface area in Euclidean space. Together with the positive mass theorem, this leads to a remarkable large scale manifestation of non-negative scalar curvature. We note that this implies that, in the examples constructed by R. Schoen and the first-named author that we described above, sufficiently large spheres in the Euclidean half-space, though evidently volume-preserving stable CMC surfaces, are  \emph{not} isoperimetric. We include the following consequence of this discussion, which sharpens \cite[Theorem 1.2]{hdiso} of J. Metzger and the third-named author: 

\begin {theorem} \label{thm:largeiso} Let $(M, g)$ be an asymptotically flat Riemannian $n$-manifold with horizon boundary, integrable scalar curvature, and positive ADM-mass. For all $V>0$ sufficiently large there is an isoperimetric region of volume $V$, i.e., there is a bounded region $\Omega_V \subset M$ of volume $V$ that contains the horizon such that 
\begin{align} 
\area (\partial \Omega_V) = \inf \{ \area (\partial \Omega) : \Omega \subset M \text{ smooth open region containing the horizon, volume } V\}.
\end{align}
The region $\Omega$ is smooth away from a thin singular set of Hausdorff dimension $\leq n-7$.
\end{theorem}

Assume now that $n = 3$ and that the scalar curvature of $(M, g)$ is non-negative. Remarkably, isoperimetric regions $\Omega_V$ exist in $(M, g)$ for \emph{all} volumes $V > 0$ in this case. This follows from a beautiful observation due to Y. Shi \cite{Shi:isoIMCF}, as we explain in Appendix \ref{sec:appendixisoallvolume}. It is natural to wonder about the behavior of $\Omega_V$ for large volumes $V>0$. For simplicity of exposition, we assume for a moment that $M$ has empty boundary. Let $\Sigma_i = \partial \Omega_{V_i}$ where $V_i \to \infty$. It has been shown in \cite{isostructure} that these surfaces either diverge to infinity as $i \to \infty$, or that alternatively a subsequence of these surfaces converges geometrically to a non-compact area-minimizing boundary $\Sigma \subset M$. In view of Theorem \ref{thm:Schoenrigidity}, the latter is impossible unless $(M, g)$ is flat $\R^3$. We arrive at the dichotomy that large isoperimetric regions in $(M, g)$ are either drawn far into the asymptotically flat end, or they contain the center of the manifold. 

\begin {corollary} Let $(M, g)$ be an asymptotically flat Riemannian $3$-manifold with non-negative scalar curvature and positive mass. Let $U \subset M$ be a bounded open subset that contains the boundary of $M$. There is $V_0 > 0$ so that for every isoperimetric region $\Omega \subset M$ of volume $V \geq V_0$, either $U \subset \Omega_V$ or $U \cap \Omega_V$ is a thin  smooth region that is bounded by the components of $\partial M$ and nearby stable constant mean curvature surfaces.
\end {corollary} 
Note that the conclusion of the corollary is wrong for flat $\R^3$. When the scalar curvature of $(M, g)$ is everywhere positive, this result was observed as Corollary 6.2 of \cite{isostructure}. The role of Theorem \ref{thm:Schoenrigidity} here is that of Theorem \ref{thm:stableminimalnocrossing} in the proof of Corollary \ref{cor:largeCMC}.

\ \\

\noindent {\bf Acknowledgments.} The first-named author wishes to express his gratitude to Richard Schoen for introducing him, with great professionality and unparalleled enthusiasm, to the mathematical challenges of general relativity. He also thankfully acknowledges the support of Andr\'e Neves through his ERC Starting  Grant. The second-named author would like to convey his deepest thanks to his advisor Simon Brendle for his invaluable support and continued encouragement. His research was supported in part by an NSF fellowship DGE-1147470 as well as the EPSRC grant EP/K00865X/1. The third-named author expresses his gratitude to Hubert Bray, Simon Brendle, Greg Galloway, Gerhard Huisken, Jan Metzger, and Richard Schoen. A part of this paper was written up during his invigorating stay of two months at Stanford University, which was supported by their Mathematical Sciences Research Center. The second- and third-named authors would also like to thank the Erwin-Schr\"odinger-Institute of the University of Vienna for its hospitality during the special program ``Dynamics of General Relativity: Numerical and Analytic Approaches" in the summer of 2011. It is a pleasure to sincerely congratulate R. Schoen on the occasion of his 65th birthday.


\section{Sheeting of volume-preserving stable CMC surfaces} 

\begin {prop} \label{prop:sheeting} Let $(M, g)$ be a homogeneously regular Riemannian $3$-manifold with non-negative scalar curvature $R \geq 0$. Assume that there is a bounded open set $O \subset M$ and a sequence 
$
\{\Sigma_k\}_{k=1}^\infty
$
of connected closed volume-preserving stable CMC surfaces in $(M, g)$ with
\begin{equation} \label{areasdiverge}
\lim_{k \to \infty} \area (O \cap \Sigma_k) = \infty.
\end{equation}
There exists a totally geodesic stable minimal immersion $\varphi : \Sigma \to M$ that does not cross itself. Moreover, $\Sigma$ with the induced metric is conformal to the plane and the ambient scalar curvature vanishes along this immersion. 
\begin {proof} 
It follows from \eqref{eqn:ChristodoulouYau} and \eqref{areasdiverge} that the mean curvatures of the surfaces tend to $0$ as $k \to \infty$. By Lemma \ref{lem:curvatureestimates}, the second fundamental forms of the surfaces are bounded independently of $k$.
Passing to a subsequence if necessary, we can find $p \in M$ such that 
\begin{equation}\label{eq:area-concentration}
\lim_{k \to \infty} \area (B_r(p) \cap \Sigma_k) = \infty
\end{equation}
for all $r > 0$. Choose base points $x_k^* \in \Sigma_k$ for the submanifolds $\Sigma_k$ with
$
\lim_{k \to \infty} x_k^* = p.
$
Passing to a convergent subsequence, we obtain a complete minimal immersion $\tilde \varphi : \tilde \Sigma \to M$ with base point $\tilde x^*$ such that $\tilde \varphi (\tilde x^*) = p$. As it is the limit of embedded surfaces, this immersion does not cross itself.
Its second fundamental form is bounded. In particular, the area of small geodesic balls in $\tilde \Sigma$ is bounded below uniformly in terms of the radius. We see from \eqref{areasdiverge} that $\tilde \Sigma$ is non-compact. 

Let $\pi : \Sigma \to \tilde \Sigma$ be the universal cover of $\tilde \Sigma$. Let $x^{*} \in \Sigma$ be a point such that $\pi(x^{*}) = \tilde x^{*}$. Consider the immersion $\varphi = \tilde \varphi \circ \pi : \Sigma \to M$.

In the argument below, we denote the second fundamental forms of the submanifolds $\Sigma_k$ and the immersion $\varphi : \Sigma \to M$ by $h_k$ and by $h$ respectively. Let $U \subset \Sigma$ be open, bounded, connected, and simply connected with $x^* \in U$. Fix $r > 0$ sufficiently small. 

Using the curvature bounds and \eqref{eq:area-concentration}, upon passing to a further subsequence, we see that there are $n(k)$ components of $B_r(p)\cap \Sigma_k$ that are geometrically close to one another, where $n(k)$ is strictly increasing in $k$. In fact, we can choose points  $x_k^1, \ldots, x_k^{n(k)} \in B_r(p) \cap \Sigma_k$ contained in these components such that $x_k^j \to p$ as $k \to \infty$ for every $j \geq 1$. Using the maximum principle, we see that for every $j\geq 1$, the submanifolds $\Sigma_k$ with respective base points $x_k^j$ converge to to an immersion which agrees with $\varphi : \Sigma \to M$ after passing to the universal cover. It follows that we can find $u_{k}^{1},\dots,u_{k}^{n(k)} : U \to \R$ such that $u_{k}^{j} \to 0$ in $C^{2}_{loc}(U)$ as $k \to \infty$ for every $j \geq 1$, and such that
\[
\Sigma_k^j = \{ \exp u_{k}^{j}(x)\nu(\varphi(x)) : x \in U\}
\]
are disjoint subsets of $\Sigma_{k}$ for every $j = 1, \ldots, n(k)$. 

Assume that there is a point in $\Sigma$ where $|h|^{2} + \scal \circ \varphi > 2 \delta$ for some $\delta > 0$. Let $U \subset \Sigma$ be a subset as above that contains this point. Fix $k \geq 1$ sufficiently large. Then, for each $j \in \{1,\dots,n(k)\}$, this implies that the surface $\Sigma_k^j$ contains a subset where $|h_k|^2 + R > \delta$ whose area is bounded below independently of $k$. Because $n(k)$ can be taken arbitrarily large, this contradicts \eqref{eqn:ChristodoulouYau}. It follows that $\varphi : \Sigma \to M$ is totally geodesic and $R \circ \varphi = 0$.

To see that $\varphi : \Sigma \to M$ is stable, it is enough to show that every bounded open subset $U \subset \Sigma$ admits a positive Jacobi function. The argument below is similar to \cite[p.\ 333]{Simon:1987}, \cite[p.\ 732]{Meeks-Rosenberg:2005}, or \cite[p.\ 493]{Meeks-Rosenberg:2006}. We may assume that $U$ is simply connected and that $x^* \in U$. By the argument above, $\Sigma_k$ contains two disjoint pieces that appear as small graphs above $U$ whose unit normals approximately point in the same direction. The defining functions of these graphs are ordered. They tend to zero in $C^2_{loc}(U)$ as $k \to \infty$. These functions satisfy the same graphical prescribed constant mean curvature equation on $U$. Hence, their difference is a positive solution of a linear uniformly elliptic partial differential equation. By the Harnack principle, the supremum and the infimum of this solution are comparable on small balls. As in \cite[p.\ 333]{Simon:1987}, we may rescale and pass to a subsequence that converges to a positive solution of the Jacobi equation on $U$.  

It follows from \cite[Theorem 3 (ii)]{Fischer-Colbrie-Schoen} that $\Sigma$ with the induced metric is conformal to the plane.
\end {proof}
\end {prop}


\section {Bounded complete stable minimal immersions}

\begin{prop} \label{prop:unboundedness}
Let $(M, g)$ be an asymptotically flat Riemannian $3$-manifold with horizon boundary. Every complete minimal immersion $\varphi : \Sigma \to M$ with uniformly bounded second fundamental form is either unbounded or an embedding of a component of the horizon. 

\begin{proof} 
Assume that the trace $\varphi (\Sigma)$ of the immersion $\varphi : \Sigma \to M$ is contained in a compact set. Let $S$ be the union of the horizon and the closure of $\varphi (\Sigma)$. There is a closed minimal surface in $M$ that contains $S$. To see this, let $r > 1$ large be such that $S \subset B_r$ and such that the  mean curvature of the coordinate sphere $S_r$ with respect to the outward pointing unit normal is bounded below by $H_0 > 0$. 

Let $H \in (0, H_0)$. Consider the functional 
\[
\Omega \mapsto \mathcal{F}_H (\Omega) = \area (\partial \Omega) - H \, \text{vol} (\Omega)
\]
on 
\[
\mathcal{A} = \{ \Omega : \Omega \subset M \text{ is open with smooth boundary and } S \subset \Omega \subset B_r\}.
\]
The curvature bounds from Lemma \ref{lem:curvatureestimates} together with the completeness of the immersion ensure that $S$ acts as an effective geometric barrier for the minimization of this functional in the following sense: There is $\delta > 0$ small depending on $H \in (0, H_0)$ such that given $\Omega \in \mathcal{A}$ with 
\[
\dist (\partial \Omega, \partial (B_r \setminus S)) < \delta
\]
there is $\tilde \Omega  \in \mathcal{A}$ with 
\[
\dist (\partial \tilde \Omega, \partial (B_r \setminus S)) \geq \delta
\]
such that 
\[
\mathcal{F}_H (\tilde \Omega) < \mathcal{F}_H (\Omega).
\]
This follows from a classical calibration argument, see for example \cite [Lemma 7.2] {Duzaar-Fuchs:1990}, based on vector fields as described in Lemma \ref{lem:effectivebarrier}. Standard arguments of geometric measure theory, see for example \cite{Duzaar-Fuchs:1990, Fuchs:1991}, imply that there is a minimizer $\Omega_H \in \mathcal{A}$ of $\mathcal{F}_H$. Its boundary $\Sigma_H = \partial \Omega_H$ is a closed hypersurface in $B_r \setminus S$ with constant (outward) mean curvature $H$ that is strongly stable, i.e., its Jacobi operator is non-negative definite. We obtain that 
\[
\area (\Sigma_H) \leq \area (S_r)
\]
from direct comparison. In conjunction with strong stability, we obtain uniform curvature estimates for $\Sigma_H$ from e.g.\, \cite {Schoen-Simon-Yau:1975} or \cite{Schoen-Simon:1981}. It follows that the Hausdorff distance between $\Sigma_H$ and the horizon tends to zero as $H \searrow 0$, since otherwise we could find (by extraction of a convergent subsequence) a closed minimal surface in $(M, g)$ that is not a component of the horizon. In particular, the trace of the original immersion is contained in a component of the horizon. Since the components are spheres, it follows that the immersion is an embedding. 
\end{proof}
\end{prop}

\begin {remark} The proof of the preceding lemma should be compared to those of \cite[Lemma 4.1]{Huisken-Ilmanen:2001} and \cite[Theorem 4.1]{GAH}. The key point is to recognize that the trace of the immersion acts as a barrier for area minimization. 
\end {remark}


\section{Top sheets} \label{sec:topsheets}

\begin{lemma} \label{lem:equator-conc}
Let $(M,g)$ be an asymptotically flat Riemannian $3$-manifold. Let
$\varphi:\Sigma\to M$
be an unbounded complete stable minimal immersion that does not cross itself. For every $\varepsilon > 0$ there is $r_{0}>0$ so that for all $r \geq r_{0}$ there is a plane $\Pi = \Pi(r)$ through the origin in the chart at infinity with 
\[
\sup \left\{  r^{-1} \dist (\varphi(x), \Pi) + |\proj_{\Pi}(\nu(x))| : x \in \varphi^{-1}(S_r)\right\}  < \varepsilon .
\]

\begin{proof} All rescalings take place in the chart at infinity. 

Suppose, for a contradiction, that for some $\varepsilon > 0$ there is a sequence  $1 < r_{k}\to \infty$ such that
\[
\sup \{   r_{k}^{-1}\dist (\varphi(x), \Pi) + |\proj_{\Pi}(\nu(x))|  : p \in \varphi^{-1}(S_{r_k})\} \geq \varepsilon 
\] 
for every plane $\Pi$ through the origin. Let $x_k^{*}\in \Sigma$ be points with $|\varphi (x_k^{*})| = r_k$. 
It follows from Proposition \ref{prop:blow-down-is-plane} that there is a plane $\Pi_1$ through the origin so that, after passing to a subsequence, the rescaled immersions 
\[
\varphi^{-1} (M \setminus \overline{ B_1}) \to \R^3 \setminus \overline{ B_{1/r_k} (0)}  \quad \text{ given by } \quad x \mapsto \varphi (x)/r_k
\]
with respective base points $x_k^*$ converge to an immersion 
\[
\varphi_1 : \Sigma_1 \to \mathbb{R}^{3} \setminus \{0\}
\]
with $\varphi_1(\Sigma_1) = \Pi_1 \setminus \{0\}$. Let $y_{k}^{*}\in\Sigma$ be points such that $\varphi(y_{k}^{*})\in S_{r_{k}}$ and 
\begin{align} \label{aux:ec}
r_{k}^{-1 }\dist(\varphi(y_{k}^{*}), \Pi_1) + |\proj_{\Pi_{1}}(\nu(y_{k}^{*}))| \geq \frac {\varepsilon}{2}.
\end{align}
By Proposition \ref{prop:blow-down-is-plane}, there is a plane $\Pi_2$ through the origin such that a subsequence of the immersions 
\[
\varphi^{-1} (M \setminus \overline{ B_1}) \to \R^3 \setminus \overline{ B_{1/r_k} (0)} \quad \text{ given by } \quad x \mapsto \varphi (x) / r_k
\]
with respective base points $y_k^*$ converges to an immersion
\[
\varphi_2 : \Sigma_2 \to \mathbb{R}^3 \setminus \{0\} 
\]
with $\varphi_2(\Sigma_2) = \Pi_2 \setminus \{0\}$. We must have that $\Pi_1 = \Pi_2$ because the original immersion does not cross itself. This contradicts \eqref{aux:ec}. 
\end{proof}
\end{lemma}


\begin{prop} \label{prop:pass-to-top}
Let $(M,g)$ be an asymptotically flat Riemannian $3$-manifold with non-negative scalar curvature. Assume that there is an unbounded complete stable minimal injective immersion
$
\varphi:\Sigma\to M. 
$
Then there is a proper such embedding. 

\begin{proof} All rescalings take place in the chart at infinity. 

By Lemma \ref{lem:equator-conc}, after a rotation of the chart at infinity, there is $r > 1$ large so that 
\[
\sup \{  \dist (\varphi(x), \Pi) : x \in \varphi^{-1}(S_r)\} \leq r/2
\]
where $\Pi = \{x^{3}=0\}$ and
\begin{align} \label{aux:passingtopsheet}
\left| \nu(x) \cdot e_3 \right| \geq \frac{1}{2}
\end{align}
for all $x \in \Sigma$ with $|\varphi(x)| = r$.

Let $x_{k}^{*}\in \Sigma$ be points such that $|\varphi(x_{k}^{*})|= r$ and 
\begin{align*} 
\lim_{k \to \infty} \varphi^3(x_{k}^{*}) = \sup \{ \varphi^3 (x) : x \in \Sigma \text{ with } |\varphi(x)| = r\}.
\end{align*}
Here, $\varphi^3 = x^3 \circ \varphi$ on $\varphi^{-1} ( M \setminus \overline B_1)$. The second fundamental form of the immersion is bounded by Lemma \ref{lem:curvatureestimates}. The pointed immersions $\varphi : \Sigma \to M$ with respective base points $x_k^*$ subconverge to an unbounded complete stable minimal immersion $\hat \varphi : \hat \Sigma \to M$ with base point $\hat x^*$ that does not cross itself and such that $\hat\varphi(\hat x^{*})  \in S_r$. 
It follows from Corollary \ref{cor:nocylinders} that $\hat\Sigma$ with the induced metric is conformal to the plane. Lemma \ref{lem:generalquotient} shows that $\hat\varphi$ is injective.
Note that 
\begin {align}\label{eqn:choosingtopsheet}
\hat \varphi^3 (\hat x^*) = \sup \{ \hat \varphi^3 (x) : \hat x \in \hat \Sigma \text{ with } |\hat \varphi(\hat x)| = r\}.
\end{align} 
Thus $\hat\varphi(\hat\Sigma)\cap S_r$ is a disjoint union of traces of complete injectively immersed curves. In view of \eqref{aux:passingtopsheet}, these curves are either infinite spirals or simple and closed. The curve containing $\hat \varphi (\hat x^*)$ is simple and closed by \eqref{aux:passingtopsheet} and \eqref{eqn:choosingtopsheet}. The preimage $\gamma$ of this curve under $\hat \varphi$ is simple and closed in $\hat \Sigma$. By the maximum principle, the image under $\hat \varphi$ of the bounded open region in $\hat \Sigma$ bounded by $\gamma$ is contained in $B_r$. Finally, a continuity argument using Lemma \ref{lem:transverse-int-sph} gives that $\hat \varphi : \hat \Sigma \to M$ is a proper embedding.  
\end{proof}
\end{prop}


\section{Proofs of main theorems}

\begin {proof} [Proof of Theorem \ref{thm:stableminimalproper}] Any non-compact, proper immersion $\varphi : \Sigma \to M$ must have unbounded trace. It follows from Corollary \ref{cor:nocylinders} that $\Sigma$ with the induced metric is conformal to the plane. The Ricci tensor of the Schwarzschild metric \eqref{eqn:Schwarzschild} is given by
\begin{equation*}
\frac{m}{|x|^3} \left( 1+ \frac {m}{2 |x|} \right)^{-2}\left( \delta_{ij} - 3  \frac{x^k x^\ell}{ |x|^2} \delta_{ik} \delta_{j \ell} \right) dx^i \otimes dx^j.
\end{equation*}
In conjunction with Lemma \ref{lem:normalbecomesradial}, we see that 
\begin{align} \label{aux:estimateRicci}
\Ric(\nu,\nu) (x) \geq \frac{m}{2}  |\varphi (x)|^{-3}
\end{align}
holds for all $x \in \Sigma$ with $|\varphi (x)|$ sufficiently large. Since the immersion is proper, it follows that the negative part of $\Ric (\nu, \nu)$ is integrable. Using the conformal invariance of the Dirichlet energy in dimension two, the logarithmic cut-off trick, and Fatou's lemma, we obtain that 
\begin{align} \label{aux:properimmersion}
\int_{\Sigma} |h|^2 + \Ric (\nu, \nu) \leq 0
\end{align}
from stability. It follows that the function 
\[
x \mapsto |\varphi(x)|^{-3}
\]
is integrable along the immersion. Using also the Gauss equation \eqref{eqn:Gaussequation} and the estimate 
\begin{align} \label{aux:estimatescalar}
\scal \circ \varphi (x) = o ( |\varphi (x)|^{-3}) \quad \text{ as } |\varphi(x)| \to \infty,
\end{align}
we see that the Gauss curvature of the immersion is integrable. Rewriting the integrand in \eqref{aux:properimmersion} using the Gauss equation in the manner of R. Schoen and S.-T. Yau, we conclude that 
\[
\frac{1}{2} \int_\Sigma |h|^2 + \scal \circ \varphi \leq \int_\Sigma K.
\] 
In particular,
\begin{align} \label{aux2:properimmersion}
0 \leq \int_\Sigma K. 
\end{align}
For $r > 1$ sufficiently large, we have that $\Sigma_r = \varphi^{-1} (B_r)$ is a smooth bounded region by Lemma \ref{lem:transverse-int-sph}. In fact, it follows from the argument in the proof of Lemma \ref{lem:transverse-int-sph} that $\Sigma_r$ is connected. The maximum principle gives that $\Sigma_r$ is simply connected. 

At this point, we argue as in \cite[Proposition 3.6]{stableCMC}, except that we use limits of pointed immersions instead of limits in the sense of geometric measure theory. By Proposition \ref{prop:blow-down-is-plane}, the geodesic curvature of the boundary of $\Sigma_r$ with respect to the induced metric is given by 
\[
\kappa = (1 + o (1))/r \quad \text{ as } r \to \infty.
\]
Moreover,\footnote{In fact, either $1 = \limsup_{r \to \infty} {\text{length} (\partial \Sigma_r)}/{2 \pi r}$ or $2 \leq \limsup_{r \to \infty} {\text{length} (\partial \Sigma_r)}/{2 \pi r} $.}
\[
\limsup_{r \to \infty} \frac{\text{length} (\partial \Sigma_r)}{2 \pi r} \geq 1.
\]
Recall that the Gauss-Bonnet formula reads
\[
\int_{\Sigma_r} K + \int_{\partial \Sigma_r} \kappa = 2 \pi .
\]
By \eqref{aux2:properimmersion}, we obtain that   
\[
\limsup_{r \to \infty} \frac{\text{length} (\partial \Sigma_r)}{2 \pi r} = 1 \quad \text{ and } \quad \int_{\Sigma} K = 0. 
\]
A modification of the argument in \cite[p.\ 209]{Fischer-Colbrie-Schoen} using the logarithmic cut-off trick in the construction of the test functions $\zeta$ shows that $K = 0$; cf. \cite[p. 11] {Carlotto}. This is incompatible with the Gauss equation \eqref{eqn:Gaussequation} and the estimates \eqref{aux:estimateRicci} and \eqref{aux:estimatescalar}. 
\end {proof}

\begin {remark} \label{rem:if} The argument from \cite{Fischer-Colbrie-Schoen} applied as in the last step of the preceding proof shows that the surface $\Sigma \subset M$ in Proposition \ref{prop:simplepmt} is intrinsically flat.
\end {remark}


\begin{proof}[Proof of Theorem \ref{thm:reduction}]
The domain $\Sigma$ with the induced metric is conformal to the plane by Corollary \ref{cor:nocylinders}. If the immersion is injective, the result follows from Proposition \ref{prop:pass-to-top}. If not, it follows from Remark \ref{rem:willbeinjective} and Lemma \ref{lem:stability-factors} that the immersion $\varphi : \Sigma \to M$ factors to an unbounded complete stable minimal immersion $\tilde \varphi : \tilde \Sigma \to M$ through a side-preserving covering $\pi : \tilde \Sigma \to \Sigma$. Note that $\tilde \Sigma$ is cylindrical by topological reasons. This is  impossible by Corollary \ref{cor:nocylinders}. 
\end{proof}


\begin{proof} [Proof of Theorem \ref{thm:stableminimalnocrossing}] This is immediate from Theorems \ref{thm:reduction} and \ref{thm:stableminimalproper}, Lemma \ref{lem:curvatureestimates}, and Proposition \ref{prop:unboundedness}.
\end{proof}


\begin {proof} [Proof of Theorem \ref{thm:Schoenrigidity}]

We first deal with the case where the boundary of $M$ is empty. 

Let $r_0 > 0$ be as in Appendix \ref{sec:cf}. Let $\rho_0 > 1$ be such that $S_\rho$ is convex for all $\rho \geq \rho_0$. Every closed minimal surface of $(M, g)$ is contained in $B_{\rho_0}$. 

Let $\Sigma \subset M$ be a connected unbounded properly embedded and separating surface that is area-minimizing with respect to $g$. Fix a component $M_+$ of the complement of $\Sigma$ in $M$ and choose a point $p \in M_+$ to the following specifications: 
\begin {itemize}
\item $B_{\rho_0}$ is disjoint from $\{x \in M : \dist_g(x, p) < 4 r\}$;
\item $r = \dist_g(\Sigma, p)/2 < r_0$;
\item $\Sigma$ intersects $\{x \in M : \dist_g(x, p) < 4 r\}$ in a single component, and the function $\dist_g( \, \cdot \, , p)$ is decreasing in the direction of the unit normal of this component that is pointing into $M_+$.

\end {itemize}

In Appendix \ref{sec:cf}, we construct a family of conformal Riemannian metrics $\{g(t)\}_{t \in (0, \epsilon)}$ on $M$ with the following properties (see also Figure \ref{Figure 1}): 
\begin {enumerate} [(i)]
\item $g(t) \to g$ smoothly as $t \to 0$;
\item $g(t) = g$ on $\{x \in M :  \dist_g (x, p) \geq 3 r \}$;
\item $g(t) \leq g$ as quadratic forms on $M$, with strict inequality on $\{x \in M : r < \dist_g (x, p) < 3 r\}$;
\item the scalar curvature of $g(t)$ is positive on $\{x \in M : r < \dist_g(x, p) < 3 r\}$;
\item the region $M_+$ is weakly mean-convex with respect to $g(t)$. 
\end {enumerate}

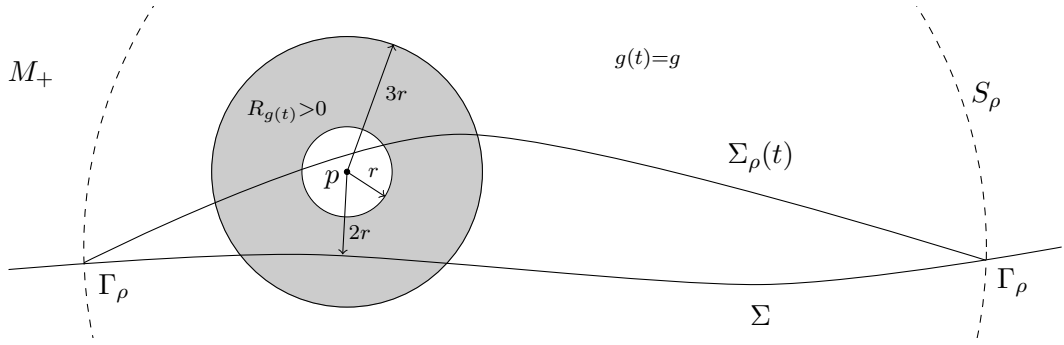
\begin{figure}[h!] 
\begin{tikzpicture}
	\filldraw [white] (-7,-1.2) rectangle (7,3.4);

	\filldraw [opacity = .2] (-2.5,1) circle (1.8);
	\filldraw [white] (-2.5,1) circle (.6);
	
	\draw (-2.5,1) circle (.6);
	\draw (-2.5,1) circle (1.8);
	
	\filldraw (-2.5,1) circle (1pt) node [shift={(-.2,-.1)}] {$p$};
	\draw [->] (-2.5,1) -- node [shift = {(.1,.15)}] {$\scriptstyle r$} +(-33:.6);
	\draw [->] (-2.5,1) -- node [shift = {(.2,-.25)}] {$\scriptstyle 2r$} +(-93:1.1);
	\draw [->] (-2.5,1) -- node [shift = {(.35,.2)}] {$\scriptstyle 3r$} +(70:1.8);

	\begin{scope}
		\clip (-7,-1.2) rectangle (7,3.2);
		\draw [dashed] (0,0) circle (6);
	\end{scope}
	
	\draw plot [smooth] coordinates {(-7,-.3) (-3,-.1) (3,-.5) (7,0)};
	\node at (3,-.9) {$\Sigma$};
	
	\draw plot [smooth] coordinates {(182:6) (-1,1.5)  (-1.7:6)};
	\node at (3,1.2) {$\Sigma_{\rho}(t)$};
	
	\node at (-5.6,-.55) {$\Gamma_{\rho}$};
	\node at (6.35,-.47) {$\Gamma_{\rho}$};
	\node at (6,2) {$S_{\rho}$};
	
	\node at (1.5,2.5) {$\scriptstyle g(t) = g$};
	\node at (-3.3,1.8) {$\scriptstyle R_{g(t)} > 0$};

	\node at (-6.7,2.3) {$M_{+}$};
\end{tikzpicture}
\caption{A diagram of the perturbed metric $g(t)$ and corresponding surface $\Sigma_{\rho}(t)$ used in the proof of Theorem \ref{thm:Schoenrigidity}.} \label{Figure 1}
\end{figure}
By taking $\epsilon >0$ smaller if necessary, we may assume that all closed minimal surfaces of $(M, g(t))$ are contained in $B_{\rho_0}$.

According to Proposition \ref{prop:simplepmt}, for all $\rho \geq \rho_0$ sufficiently large, the intersection of $\Sigma$ with $S_\rho$ is transverse in a nearly equatorial circle. We denote this circle by $\Gamma_\rho = \Sigma \cap S_\rho$. Consider all properly embedded surfaces in $M$ that have boundary $\Gamma_\rho$ and which together with $\Sigma \cap B_\rho$ bound an open subset of $M_+ \cap B_\rho$. Using (v) and standard existence results from geometric measure theory, we see that among all these surfaces there is one  --- call it $\Sigma_\rho (t)$ ---  that has least area with respect to $g(t)$. This surface is disjoint from $M_+ \cap S_\rho$ by convexity. It has one component with boundary $\Gamma_\rho$. Its other components are closed minimal surfaces in $(M, g(t))$ that are disjoint from $\{x \in M : \dist_g(x, p) < 3 r\}$. Importantly though, $\Sigma_\rho(t)$ intersects $\{x \in M : \dist_g (x, p) < 3 r\}$, since otherwise,
\begin{eqnarray} \label{eqn:otherwise}
\area_g (\Sigma_\rho(t)) = \area_{g(t)} (\Sigma_\rho(t)) \leq \area_{g(t)} (\Sigma \cap B_\rho) < \area_g (\Sigma \cap B_\rho).
\end{eqnarray}
The strict inequality holds on account of (iii) and because $\Sigma$ intersects $\{x \in M : \dist_g(x, p) < 3 r\}$. Observe that \eqref{eqn:otherwise} violates the area-minimizing property of $\Sigma$ with respect to $g$.

Using standard convergence results from geometric measure theory, we now find a connected unbounded properly embedded separating surface $\Sigma(t) \subset M$ as a subsequential geometric limit of $\Sigma_\rho(t)$ as $\rho \to \infty$. By construction, $\Sigma(t)$ is contained in $M_+ \cup \Sigma$ where it is area-minimizing with respect to $g(t)$. Moreover, $\Sigma (t)$ intersects $\{x \in M : \dist_g (x, p) \leq 3 r\}$. If $\Sigma (t)$ intersects $\{x \in M : \dist_g(x, p) < 3 r\}$, then it also intersects $\{x \in M : \dist_g(x, p) \leq r\}$ because of (iv) and Proposition \ref{prop:simplepmt}. Passing to a subsequential geometric limit as $t \to 0$, we obtain a connected unbounded properly embedded separating surface $\Sigma_+ \subset M$ that is contained in $M_+ \cup \Sigma$ where it minimizes area with respect to $g$. Using now the area-minimizing property of $\Sigma$, we see that $\Sigma_+$ is in fact area-minimizing in all of $M$. Note that $\Sigma$ intersects $\{x \in M : \dist_g(x, p) < 3 r\}$ while it is disjoint from $\{x \in M : \dist_g(x, p) \leq r\}$. It follows from the maximum principle that $\Sigma$ and $\Sigma_+$ are disjoint. 

Repeating this argument with choices of $p \in M_+$ converging to a fixed point on $\Sigma$, we obtain a sequence of totally geodesic intrinsically flat planes in $M$ (see Proposition \ref{prop:simplepmt} and Remark \ref{rem:if}) along which the ambient scalar curvature vanishes and that converge to $\Sigma$ from one side. Proceeding as in \cite[p.\ 333]{Simon:1987} but using the first variation of the second fundamental form instead of the Jacobi equation, cf. \cite{Anderson-Rodriguez:1989} and \cite{Liu:2013}, we obtain a positive function $f \in C^\infty (\Sigma)$ such that 
\begin{align} \label{eqn:firstvariationsff}
(\nabla^2_\Sigma f) (X, Y) + Rm (\nu, X, Y, \nu) f = 0
\end{align}
for all tangent fields $X, Y$ of $\Sigma$. Here, $\nu$ is a unit normal field of $\Sigma$. Tracing this equation and using that $\Ric (\nu, \nu) = 0$ (this follows from the Gauss equation), we obtain that 
\[
\Delta_\Sigma f = 0.
\] 
It follows that $f$ is a positive constant. Going back to the original equation \eqref{eqn:firstvariationsff}, we see that $\Rm (\nu, X, Y,\nu) = 0$ whenever $X, Y$ are tangential to $\Sigma$. The Codazzi equation implies that $\Rm (X, Y, Z, \nu) = 0$ provided that $X, Y, Z$ are tangential, and the Gauss equation gives that $\Rm (X, Y, Z, W) = 0$ whenever $X, Y, Z, W$ are tangential. It follows that the ambient curvature tensor vanishes along $\Sigma$. 

We may repeat this argument, beginning with any surface $\Sigma_{+}$ constructed as above. It follows that an open neighbourhood of $\Sigma$ in $(M, g)$ is flat and in fact isometric to standard $\R^2 \times (- \epsilon, \epsilon)$ for some $\epsilon > 0$. Moreover, the surfaces in $M$ that correspond to $\R^2 \times \{z\}$ where $z \in (- \epsilon, \epsilon)$ are all area-minimising. Using standard compactness properties of such surfaces and a continuity argument, we conclude that $(M, g)$ is isometric to flat $\R^3$. 

We now turn to the general case where $M$ has boundary. Consider $\Omega \in \mathcal{F}$ with non-compact area-minimizing boundary $\Sigma \subset M$. The unique non-compact component $\Sigma_0 \subset M$ of $\Sigma$ is a separating surface. Let $M_-$ and $M_+$ denote the two components of its complement in $M$. Note that the interior of $\Omega \cap M$ agrees with either $M_-$ (Case 1) or $M_+$ (Case 2) outside of $B_{\rho_0}$. The proof that $g$ is flat in $M_+$ proceeds exactly as above, except for the following change. In Case 1, we let $\Sigma_\rho (t)$ have least area among properly embedded surfaces with boundary $\Gamma_\rho$ that bound together with $\Sigma_0 \cap B_\rho$ in $M_+ \cap B_\rho$ and relative to $M_+ \cap \partial M$. In Case 2, we let $\Sigma_\rho(t)$ have least area among properly embedded surfaces with boundary $\Gamma_\rho$ that bound together with $M_+ \cap S_\rho$ in $M_+ \cap B_\rho$ and relative to $M_+ \cap  \partial M$. Theorem \ref{thm:Schoenrigidity} follows upon switching the roles of $M_-$ and $M_+$.
\end {proof}

\begin {remark} The use of the conformal change of metric in this proof is inspired by an idea of G. Liu in his classification of complete non-compact Riemannian $3$-manifolds with non-negative Ricci curvature \cite{Liu:2013}. The observation \eqref{eqn:otherwise} is crucial in the proof of Theorem \ref{thm:Schoenrigidity}, as we use it to be sure that the surfaces $\Sigma_\rho(t)$ do not run off as $\rho \to \infty$. This observation is not needed for Theorem \ref{thm:rigidityslabs} below, since the solutions of Plateau problems considered in the proof cannot escape the slab as we pass to the limit. We point out that at a related point in the work of M. Anderson and L. Rodr\'iguez \cite{Anderson-Rodriguez:1989}, their assumption of non-negative Ricci curvature is used tacitly in their delicate estimation of comparison surfaces \cite[(1.5)]{Anderson-Rodriguez:1989}. 
\end {remark}

\begin{proof} [Proof of Theorem \ref{thm:rigidityslabs}] Since $(M, g)$ has horizon boundary, $M$ is diffeomorphic to the complement of a finite union of open balls with disjoint closures in $\R^3$. Let $\Omega \subset M$ be the connected region bounded by two disjoint unbounded properly embedded complete minimal surfaces $\Sigma_0, \Sigma_1 \subset M$. By solving a sequence of Plateau problems in $\Omega \cap B_r$ with boundary on $\Omega \cap S_r$ and passing to a subsequential geometric limit as $r \to \infty$, we obtain an unbounded properly embedded boundary $\Sigma \subset M$ that is contained in $\bar \Omega$ where it is area-minimizing with respect to $g$. 
In particular, every component of $\Sigma$ is a stable minimal surface. By the maximum principle, if such a component intersects with $\Sigma_0$ or $\Sigma_1$, then it coincides with the respective surface. By Theorem \ref{thm:Carlotto}, every unbounded component is a totally geodesic flat plane along which the ambient scalar curvature vanishes. We may now proceed as in the proof of Theorem \ref{thm:Schoenrigidity}.
\end{proof}


\begin{proof}[Proof of Theorem \ref{thm:largeCMC}]
Assume that there exist a compact set $C \subset M$ and closed volume-preserving stable CMC surfaces $\Sigma_k \subset M$ with $\Sigma_{k}\cap C \not =\emptyset$ and $\area(\Sigma_{k})\to\infty$. Suppose that 
\[
\sup_{k} \area(B_r\cap \Sigma_{k}) < \infty
\]
for every $r > 1$. Using the methods from \cite{stableCMC} we find an unbounded complete stable minimal surface $\Sigma \subset M$ that is properly embedded. (In fact, the surface has quadratic area growth.) In conjunction with Theorem \ref{thm:Carlotto}, this contradicts our hypothesis.\footnote{The proof of Theorem \ref{thm:Carlotto} simplifies considerably for surfaces with quadratic area growth. Indeed, the arguments in \cite[Sections 3 and 4]{stableCMC} show that $\int_{\Sigma}K = 0$. It follows from \cite[p.\ 209]{Fischer-Colbrie-Schoen} that $\Sigma$ is flat with its induced metric. Lemma \ref{lem:normalbecomesradial} is quite elementary for surfaces with quadratic area growth, see the argument in \cite[Lemma 3.5]{stableCMC}. Finally, the Gauss equation rearrangement argument applied in the manner of R. Schoen and S.-T. Yau leads to a contradiction.} 

Assume now that
\[
\sup_{k} \area(B_r\cap \Sigma_{k}) = \infty,
\]
for some $r > 1$. Using Proposition \ref{prop:sheeting} we obtain a complete stable minimal immersion $\varphi: \Sigma \to M$ that does not cross itself and where $\Sigma$ with the induced metric is conformal to the plane. Such an immersion must be unbounded by Proposition \ref{prop:unboundedness} and the fact that the components of the horizon are spheres. This contradicts Theorem \ref{thm:reduction}. \end{proof}

\begin {remark} We remark that in the preceding proof, because the immersion at hand is totally geodesic, the argument for ``passing to the top sheet" simplifies. Indeed, we obtain the estimate \[|h_{\delta}(x)|\leq O(|\varphi(x)|^{-2})\] as $|\varphi (x)| \to \infty$ for the Euclidean second fundamental form $h_\delta$ of the immersion. This can be integrated up at infinity to show that the immersion is essentially a union of multi graphs above a fixed plane outside a large compact set. 
\end {remark}


\begin{proof} [Proof of Theorem \ref{thm:largeiso}]

Assume that for a sequence $V_i \to \infty$ there is no isoperimetric region $(M, g)$ of volume $V_i$. The argument in the proof of \cite[Theorem 1.2]{hdiso} (see also \cite[Theorem 2]{Nardulli:existence} for a much more general version of this line of argument in the case where the horizon is empty) shows that there is a minimizing sequence for 
\begin{align} \label{aux:iso}
\inf \{ \area (\partial \Omega) : \Omega \subset M \text{ smooth open region of volume $V_i$ containing the horizon}\}
\end{align}
consisting of a divergent sequence of coordinate balls of radii $r_j (V_i)$ and a \emph{residual} isoperimetric region $\tilde \Omega (V_i)$, and that the volumes of these residual regions diverges as $i \to \infty$. Moreover, we have that 
\[
\lim_{j \to \infty} \frac{n-1}{r_j (V_i)} =  H(V_i)
\]
where $ H(V_i)$ is the (outward) mean curvature scalar of $\partial \tilde \Omega (V_i)$. Let $\tilde r(V_i) = 2 / H(V_i)$. The blow-down argument in \cite{hdiso} shows that $\tilde \Omega (V_i)$ is close to a coordinate ball of radius $1$ upon rescaling by $\tilde r(V_i)$ when $i$ is sufficiently large. We conclude that \eqref{aux:iso} is almost achieved by the union of two large disjoint coordinate balls of comparable radii provided $i$ is sufficiently large. This contradicts the Euclidean isoperimetric inequality. \end{proof}


\appendix 

\section {Basic notions and conventions} \label{sec:definitions}

Consider a complete Riemannian $3$-manifold $(M, g)$, possibly with boundary. 

We say that $(M, g)$ is \emph{asymptotically flat} if there are a compact subset $K \subset M$ and a chart 
\[
M \setminus K \cong \{x \in \R^3: |x| > \frac{1}{2}\}
\]
so that the components of the metric tensor have the form 
\begin{align*}
g_{ij} = \delta_{ij} + b_{ij}
\end{align*} 
where
\[
|b_{ij}|  + |x||\partial_k b_{i j} | + |x|^{2}|\partial_{k } \partial_\ell  b_{ij}| = o (1) \quad \text{ as }  x \to \infty.
\]
Such a chart is called a \emph{structure at infinity}. We always fix such a chart when introducing an asymptotically flat Riemannian manifold and refer to it as the \emph{chart at infinity}. We also define a smooth positive function 
\[
|\cdot| : M \to (0, \infty)
\] 
that coincides with the Euclidean distance from the origin in $\R^3 \setminus B_1(0)$ in the above chart and which on $K$ is bounded by $1$. Given $r > 1$, we let 
\[
B_r = \{p \in M : |p| < r\} \quad \text{ and } \quad S_r = \{p \in M : |p| = r\}.
\]
If the scalar curvature of $(M, g)$ is integrable, then the limit
\begin{align} \label{def:mass}
\lim_{r \to \infty} \frac{1}{16 \pi} \int_{\{|x| = r\}}  \sum_{i, j = 1}^3 \left(\partial_i g_{ij} - \partial_j g_{ii} \right) \frac{x^j}{|x|}
\end{align}
exists. It is independent of the choice of structure at infinity \cite{Bartnik:1986} and called the \emph{ADM-mass} (after R. Arnowitt, S. Deser, and C. W. Misner \cite{ADM:1961}) of the asymptotically flat manifold $(M, g)$.  

We say that $(M, g)$ asymptotically flat has \emph{horizon boundary} if its only closed minimal surfaces are the components of its boundary. It is known that the boundary components of such $(M, g)$ are area-minimizing spheres. Moreover, $M$ is diffeomorphic to the complement of a finite union of open balls with disjoint closures in $\R^3$. See \cite[Lemma 4.1]{Huisken-Ilmanen:2001} and the references therein. 

Let $m \in \R$. We say that $(M, g)$ is \emph{asymptotic to Schwarzschild} with mass $m$ if there exists a chart as above such that
\begin{equation}\label{eq:asymp-Schw}
g_{ij} = \left(1 + \frac{m }{2 |x|}\right)^4 \delta_{ij} + c_{ij}
\end{equation} 
where
\[
|x| |c_{ij}|  + |x|^{2}|\partial_k c_{i j} | + |x|^{3}|\partial_{k} \partial_{\ell}  c_{ij}| = o (1) \quad \text{ as } x \to \infty.
\]

We say that an immersion $\varphi : \Sigma \to M$ \emph{does not cross} itself if given $x_1, x_2 \in \Sigma$ with $\varphi (x_1) = \varphi (x_2)$ there are $U_1, U_2 \subset \Sigma$ open with $x_1 \in U_1$ and $x_2 \in U_2$ such that $\varphi (U_1) = \varphi (U_2)$ and so that the restrictions of $\varphi : \Sigma \to M$ to $U_1$ and $U_2$ are embeddings. 

The concept of ``immersions that do not cross themselves" arises naturally when studying limits of injective immersions of co-dimension one. 


Consider a two-sided immersion $\varphi : \Sigma \to M$ of a boundaryless surface $\Sigma$ with unit normal $\nu : \Sigma \to T M$. 

Below, we use $\Ric$ and $R$ to denote the ambient Ricci tensor and scalar curvature, we write $H$ and $h$ for the (scalar) mean curvature and the second fundamental form of the immersion with respect to the designated unit normal, we denote by $K$ the Gauss curvature of the induced metric $\varphi^* g$ on $\Sigma$, and we compute gradients and lengths and perform integration with respect to the induced metric. 
  
Recall that $\varphi : \Sigma \to M$ is a \emph{stable minimal immersion} if its mean curvature vanishes and 
\begin{align} \label{eqn:stabilityminimal}
\int_\Sigma |\nabla u|^2 \geq \int_\Sigma (|h|^2 + \Ric (\nu, \nu)) u^2 \quad \text{ for all } \quad u \in C^\infty_c(\Sigma).  
\end{align}
Such immersions arise in area minimization; cf. Appendix \ref{sec:variationformulae}. 

Recall that $\varphi : \Sigma \to M$ is a \emph{volume-preserving stable CMC immersion} if its mean curvature is constant and 
\[
\int_\Sigma |\nabla u|^2 \geq \int_\Sigma (|h|^2 + \Ric (\nu, \nu)) u^2 \quad \text{ for all } \quad u \in C^\infty_c(\Sigma) \text{ with } \int_\Sigma u = 0.  
\] 
Such immersions arise in area minimization with a (relative) volume constraint, i.e. in the isoperimetric problem; cf. Appendix \ref{sec:variationformulae}. 

Finally, recall the Gauss equation 
\begin{align} \label{eqn:Gaussequation}
R \circ \varphi = 2 K + |h|^2 - H^2 + 2 \Ric (\nu, \nu).
\end{align}

We emphasize that in this paper, we adopt the convention that constant mean curvature immersions with non-zero mean curvature and stable minimal immersions are \emph{by definition} two-sided. The immersions considered here are all of co-dimension one. The domain of a complete immersion is connected by definition.

The notion of convergence for pointed immersions and compactness results in the presence of uniform curvature bounds are used throughout the paper and are reviewed in Appendix \ref{sec:convergenceimmersions}.


\section {A compactness result for pointed immersions} \label{sec:convergenceimmersions}

For a proof of the following compactness result, see \cite{Cooper}. 

\begin{lemma} [Limits of immersions] \label{lem:limitsofimmersions}
Let $(M, g)$ be a complete Riemannian manifold. Let  
\[
\{\varphi_k : \Sigma_k \to M\}_{k =1}^\infty
\] 
be a sequence of complete constant mean curvature immersions such that
\[
\sup_{k } \sup_{x \in \Sigma_k} |h_k (x)| < \infty.
\]
Assume that there are points $x_k^* \in \Sigma_k$ such that the limit
\[
\lim_{k \to \infty} \varphi_k (x_k^*)
\]
of points in $M$ exists. There is a complete constant mean curvature immersion 
\[
\varphi : \Sigma \to M
\]
and a point $x^* \in \Sigma$ so that a subsequence of the immersions 
\[
\varphi_{k} : \Sigma_k \to M \quad \text{ with base points } \quad x_k^*
\]
converges to 
\[
\varphi: \Sigma \to M \quad \text{ with base point } \quad x^*
\]
 in the sense of pointed immersions.
By this we mean that the following holds up to passing to a subsequence. Let $\nu$ be a unit normal field of $\varphi$. There are bounded open subsets $U_k \subset \Sigma_k$ and $V_k \subset \Sigma$ with 
\[
x_k^* \in U_k 
\] 
and 
\[
x^{*} \in V_1 \subset V_2 \subset \ldots \quad \text{ and } \quad \Sigma = \bigcup_{k=1}^\infty V_k 
\]
as well as diffeomorphisms 
\[
\psi_k : V_k \to U_k
\]
and functions $u_k \in C^\infty (V_k)$ with
\[
u_k \to 0 
\]
in $C^\infty_{loc}(V_\ell)$ as $\ell \leq k \to \infty$ for every $\ell \geq 1$ and
\[
\psi^{-1}_k (x_k^*) \to x^*
\] 
such that
\[
(\varphi_k \circ \psi_k )(x)  = \exp_{\varphi (x)} ( u_k (x)  \nu (x) )
\]
for all $x \in V_k$.  
\end{lemma}


\section{Rigidity of stable minimal cylinders} \label{app:rigidity} 

The result in the following proposition was established under the additional hypothesis that the Gauss curvature of the immersion be integrable in \cite[Theorem 3 (ii)]{Fischer-Colbrie-Schoen} and left as an open problem in \cite[Remark 2]{Fischer-Colbrie-Schoen}. Solutions have been proposed in \cite{Schoen-Yau:1982, Miyaoka:1993, Berard-Castillon, Reiris}.\footnote{It seems to us that the proof given in \cite{Schoen-Yau:1982} ``only" shows that there are no stable minimal immersions of the cylinder into $(M, g)$ if the ambient scalar curvature is positive; see the argument given in \cite[top of p. 216]{Schoen-Yau:1982} and also the sentence after the statement of their Theorem 2. In the proof given in \cite{Miyaoka:1993}, consider the integral over the ball $B_r$ at the bottom of page 292. In the evaluation of this integral using conformal invariance as suggested on the next page, we do not see how the geometry of the ``conformally changed" domain is controlled so that the ``order" of the test functions on the cylinder carries over.} Here we present a short proof based on a result by D. Fischer-Colbrie. 

\begin {prop} \label{prop:rigidcylinders} Let $(M, g)$ be a $3$-dimensional Riemannian manifold with non-negative scalar curvature $\scal \geq 0$. Let $\varphi : \Sigma \to M$ be a complete stable minimal immersion such that $\Sigma$ with the induced metric is conformal to the cylinder. Then the immersion is totally geodesic, the induced metric is flat, and $R \circ \varphi = 0$. 

\begin {proof}
According to \cite[Proposition 1] {Fischer-Colbrie:1985}, there is a smooth function $u > 0$ on $\Sigma$ such that 
\[
- \Delta u + Ku = \frac 12 (|h|^2 + \scal \circ \varphi) u 
\]
where $K$ and $\Delta$ are respectively the Gauss curvature and the non-positive definite Laplace-Beltrami operator of the induced metric $\varphi^* g$ on $\Sigma$ and where $|h|^2$ denotes the sum of squares of principal curvatures of the immersion. Theorem 1 in \cite{Fischer-Colbrie:1985} ensures that the conformally related metric $u^2 \,  \varphi^*g$ is complete. The Gauss curvature of this metric is given by 
\begin{equation} \label{aux:rigidity}
\frac{1}{u^2} \left( \frac{|\nabla u|^2}{u^2} + K - \frac{\Delta u}{u} \right) = \frac{1}{u^2} \left( \frac{|\nabla u|^2}{u^2} + \frac{|h|^2 + \scal \circ \varphi}{2}\right)
\end{equation}
where all geometric operations are with respect to the original induced metric. In particular, it is non-negative. It follows from the splitting theorem of J. Cheeger and D. Gromoll \cite{Cheeger-Gromoll} that $u^2 \, \varphi^* g$ is flat.\footnote{In fact, the relevant two-dimensional case of the splitting theorem is due to S. Cohn-Vossen.} The claim now follows from \eqref{aux:rigidity}.
\end {proof}
\end {prop}


\begin {corollary} \label {cor:nocylinders} Let $(M, g)$ be an asymptotically flat Riemannian $3$-manifold with non-negative scalar curvature. Let $\varphi : \Sigma \to M$ be an unbounded complete stable minimal immersion. Then $\Sigma$ with the induced metric is conformal to the plane. 
\begin{proof} Else, by \cite[Theorem 3 (ii)]{Fischer-Colbrie-Schoen}, $\Sigma$ with the induced metric is conformal to the cylinder. By Proposition \ref{prop:rigidcylinders}, the induced metric is flat and the immersion is totally geodesic. This implies the existence of simple closed geodesics far out, contradicting asymptotic flatness. 
\end {proof}
\end {corollary}


\section {Geometry of volume-preserving stable CMC immersions}

\begin {lemma} [\protect{\cite{Christodoulou-Yau:1988}}] \label{lem:ChristodoulouYau} Let $(M, g)$ be a Riemannian $3$-manifold and 
\[
\varphi : \Sigma \to M
\]
be a connected closed volume-preserving stable CMC immersion. Then 
\begin{align} \label{eqn:ChristodoulouYau}
\int_\Sigma H^2 + 2 |h|^2 + 2 (\scal \circ \varphi) \leq 64 \pi
\end{align}
where $R$ is the scalar curvature of $(M, g)$ and $H$ and $h$ denote the mean curvature and the second fundamental form of the immersion respectively. If $\Sigma$ is a sphere, then the bound on the right hand side may be lowered to $48 \pi$.
\end {lemma}


\begin{lemma} [\protect{Cf. \cite[Theorem 7]{Ye:1996} and also \cite[Proposition 2.2]{stableCMC}}] \label{lem:curvatureestimates} 
Let $(M, g)$ be a homogeneously regular Riemannian $3$-manifold. Let $\alpha > 0$. There is a constant $\beta > 0$ with the following property. Let $\varphi : \Sigma \to M$ be a complete volume-preserving stable CMC immersion whose mean curvature satisfies $|H| \leq \alpha$. Then  
\[
\sup_{x \in \Sigma}  |h(x)| \leq \beta
\]
where $h$ denotes the second fundamental form of the immersion. 
\end{lemma}


\begin{lemma} [\protect{\cite[Proposition 2.3]{stableCMC}}] \label{lem:curvatureestimates-with-decay} 
Let $(M, g)$ be an asymptotically flat Riemannian $3$-manifold. Let $C \subset M$ be compact and $\alpha > 0$. There is a constant $\beta > 0$ with the following property. Let $\varphi : \Sigma \to M$ be a complete volume-preserving stable CMC immersion whose mean curvature satisfies $|H| \leq \alpha$ and such that 
\[
\varphi (\Sigma) \cap C \neq \emptyset.
\]
Let $h$ denote the second fundamental form of the immersion. Then  
\[
\sup_{x \in \Sigma} |\varphi (x)| |h(x)| \leq \beta.
\]
\end{lemma}


\section{Asymptotic behavior of stable minimal immersions}

In this appendix, we investigate the qualitative behavior of the part of a stable minimal immersion that extends into the end of an asymptotically flat Riemannian $3$-manifold. 

The following result due to R. Gulliver and H.B. Lawson \cite{Gulliver-Lawson:1986} extends the classical result of D. Fischer-Colbrie and R. Schoen \cite{Fischer-Colbrie-Schoen}, M. do Carmo and C.K. Peng \cite{doCarmoPeng}, as well as A. V. Pogorelov \cite{Pogorelov} to the possible inclusion of an isolated singularity. 

\begin{lemma} \label{lem:Bernstein} 
Let 
\[
\varphi : \Sigma \to \mathbb{R}^3 \setminus \{0\}
\]
be a connected stable minimal immersion that is complete\footnote{In other words, every sequence $\{x_i\}_{i=1}^\infty \subset \Sigma$ that is Cauchy with respect to the induced Riemannian distance either has a limit in $\Sigma$ or is such that $|\varphi (x_i)| \to 0$.} away from the origin. Then $\varphi (\Sigma)$ is a plane.
\end{lemma}


In conjunction with Lemma \ref{lem:curvatureestimates-with-decay}, we obtain the following 

\begin{lemma}\label{lem:improved-decay}
Let $(M,g)$ be an asymptotically flat Riemannian $3$-manifold. Let
\[
\varphi : \Sigma \to M
\]
be a complete stable minimal immersion. Then
\[
|\varphi (x)| |h(x)| = o(1)
\]
as $|\varphi (x)| \to \infty$. 
\end{lemma}


The following lemma shows that complete stable minimal immersions in asymptotically flat $3$-manifolds have transverse intersection with all sufficiently large coordinate spheres. It is based on the ideas of B. White \cite[p.\ 251]{White:curvature-est} who observed a similar result for surfaces that are properly embedded in $B_{1}(0)\setminus\{0\}$ (see also the work of W.\ Meeks, J.\ P\'erez, and A.\ Ros \cite[Lemma 4.1]{Meeks-Perez-Ros:2013}). The generality we require causes a complication that is not present in \cite{White:curvature-est,Meeks-Perez-Ros:2013}. Specifically, we need to address the failure of the Palais-Smale condition (due to lack of properness) in the proof of the mountain pass lemma. Our reasoning here may be of some independent interest. 

\begin{lemma} \label{lem:transverse-int-sph}
Let $(M,g)$ be an asymptotically flat Riemannian $3$-manifold. Let $\varphi : \Sigma \to M$ be a complete stable minimal immersion. There is $r_0>1$ so that the immersion is transverse to the centered coordinate sphere $S_r$ for every $r\geq r_{0}$.
\end{lemma}
\begin{proof} 
We work in the coordinate chart 
\[
M \setminus \overline {B_1} \cong \mathbb{R}^3 \setminus \overline{ B_1(0)}
\]
at infinity. First, recall the elementary estimate 
\[
|\varphi(x)|^2 |h_\delta (x) - h_g (x)| \leq c\left( |\varphi(x)||h_{g} (x)| + 1\right)
\]
on $\varphi^{-1}(M \setminus \overline {B_1})$ which holds for a constant $c > 0$ that is independent of the immersion. Here we use $h_{g}$ and $h_{\delta}$ to denote the scalar-valued second fundamental forms with respect to the ambient metrics $g$ and $\delta$ respectively. Using also Lemma \ref{lem:improved-decay} we obtain that 
\[
|\varphi(x)||h_{\delta} (x)| =o(1)
\]
as $|\varphi(x)| \to \infty$. Let $f : \Sigma \to \R$ denote the function given by 
\[
x \mapsto |\varphi(x)|^2.
\]
Given $x \in \varphi^{-1} (M \setminus \overline {B_1})$ and $v \in T_{x}\Sigma$ we have that 
\[
\frac{1}{2}(\partial^2_{\Sigma} f)(x)(v,v) = |v|^{2} - h_{\delta} (x)(v,v) (\nu_\delta \cdot \varphi(x))
\]
where $\partial^2_\Sigma f$ and $\nu_\delta$ are respectively the Hessian of $f$ and the normal of the immersion, both take with respect to metric induced on $\Sigma$ by the ambient Euclidean metric. We obtain the convexity estimate 
\begin{equation}\label{eq:strict-convex-est-dist-sq-app}
(\partial^2_{\Sigma} f)(x)(v,v) \geq |v|^2
\end{equation}
provided $|\varphi(x)|$ is sufficiently large. In particular, the critical points of $f$ on $\varphi^{-1} (M \setminus \overline {B_r})$ are strict local minima provided $r > 1$ is sufficiently large.

In what follows, we rework the proof of the mountain pass lemma as presented in e.g. \cite[pp.\ 74--76]{Struwe:variationalMethods} or \cite [pp. 332--334]{Jost:2011}.

Let $y \in \varphi^{-1} (M \setminus \overline{B_r})$ be a critical point of $f$. We let $\Lambda$ denote the collection of all continuous paths 
\[
[0,1] \mapsto  \varphi^{-1} (M \setminus \overline{B_r})
\]
with the property that $|\varphi (\gamma(0))| = r$ and $\gamma(1) = y$. Let 
\[
\alpha = \inf_{\gamma\in\Lambda} \, \sup_{t\in[0,1]} f(\gamma(t)).
\]
Note that
\[
r^2 <  f(y) < \alpha. 
\]
Choose paths $\gamma_m \in \Lambda$ such that  
\[
\alpha = \lim_{m\to\infty} \sup_{t\in[0,1]}f(\gamma_m(t)). 
\]
Consider the quantity
\begin {equation} \label {aux:mountainpassquantity}
\lim_{\delta \to 0} \liminf_{m \to \infty} \inf \{ |(\partial_\Sigma f) (x)| : x \in I (m, \delta)\}
\end{equation}
where 
\[
I (m, \delta) =   \left\{x \in \Sigma :  \text{ there is } t \in [0, 1] \text{ such that } \dist_\Sigma (x, \gamma_m(t)) < \delta  \text{ and } |f (\gamma_m(t)) - \alpha | < \delta \right\}.
\]
Here, 
\[
\dist_\Sigma : \Sigma \times \Sigma \to \R
\]
is the Riemannian distance on $\Sigma$ with respect to the metric induced on $\Sigma$ by the ambient Euclidean metric.  If the quantity in \eqref{aux:mountainpassquantity} vanishes, then --- possibly upon passing to a subsequence --- there exist $t_m \in [0, 1]$ so that 
\[
f(\gamma_m(t_m)) \to \alpha \quad \text{ and } \quad (\partial_\Sigma f) (\gamma_m(t_m)) \to 0
\]
contradicting the choice of the paths $\gamma_m$ in view of the strict convexity estimate \eqref{eq:strict-convex-est-dist-sq-app} and the curvature estimates. Assume now that the quantity in \eqref{aux:mountainpassquantity} is bounded below by $\varepsilon >0$. Fix $\delta \in (0, 1)$ satisfying $\alpha > 2 \delta + r^2$. Up to subsequences, we have that 
\[
|(\partial_{\Sigma} f)(x)| \geq \varepsilon
\]
for all $x \in \Sigma$ for which there is $t \in [0, 1]$ verifying
\[
|f( \gamma_m(t)) - \alpha| < \delta \quad \text{ and } \quad \dist_\Sigma (\gamma_m(t), x) < \delta.
\]
Let $\chi \in C^{\infty}(\R)$ be a function such that $0 \leq \chi \leq 1$ everywhere, which is one on the inverval $[\alpha-\delta,\alpha+\delta]$, and which vanishes away from the interval $(\alpha-2\delta,\alpha+2\delta)$. The length of the vector field 
\[
x \mapsto -\chi(f(x)) (\partial_{\Sigma}f)(x)
\]
is bounded by $2 (\alpha + 2 \delta)$. Owing to the curvature estimates, its flow exists for all time. Let $\Phi_s : \Sigma \to \Sigma$ denote the time $s$ diffeomorphism generated by this vector field. Note that $\Phi_{s} \circ \gamma_{m}\in \Lambda$. As in the standard proof of the mountain pass lemma, we conclude that 
\[
\lim_{s\to\infty}\sup_{t\in[0,1]}f(\Phi_{s}(\gamma_{m}(t)) \leq  \max \left\{\alpha - \delta, \sup_{t\in[0,1]}f(\gamma_{m}(t)) - \frac{\delta \, \varepsilon}{16 \alpha}\right\}.
\]
This contradicts the choice of $\gamma_{m}$.
\end{proof}


The following two results are obtained from Lemma \ref{lem:transverse-int-sph} in a straightforward manner.

\begin{prop} \label{prop:blow-down-is-plane}
Let $(M,g)$ be an asymptotically flat Riemannian $3$-manifold and 
$
\varphi : \Sigma \to M
$
an unbounded complete stable minimal immersion. Let $\{x_k^{*}\}_{k=1}^\infty \subset \Sigma$ be points with 
\[
1 < r_k = |\varphi(x_{k}^{*})| \to\infty \quad \text{ as } k \to \infty.
\]
Consider the pointed immersion 
\[
\varphi^{-1}( M \setminus \overline{ B_1}) \to \R^3 \setminus \overline{ B_{1/r_k} (0) }\quad \text{ given by } \quad x \mapsto \varphi(x) /r_k
\]
with base point $x_k$. Here we use the chart at infinity to identify $M \setminus \overline{ B_1} \cong \R^3 \setminus \overline{ B_1(0)}$.
The trace of every subsequential limit of these pointed immersions is a plane through the origin. 
\end{prop}


\begin {lemma} \label{lem:normalbecomesradial} Let $(M,g)$ be an asymptotically flat Riemannian $3$-manifold. Let $\nu$ be a unit normal field of a complete stable minimal immersion $\varphi : \Sigma \to M$.
Then 
\[
\nu(x) \cdot \frac{\varphi (x)}{|\varphi (x)|} \to 0 \quad \text{ as } |\varphi(x)| \to \infty.
\] 
\end {lemma}


\section {Quotients of immersions} \label{app:quot-imm}

In this appendix, we collect observations on quotients of minimal immersions that do not cross themselves. The first two lemmas are elementary.

\begin {lemma} Let $(M, g)$ be a Riemannian manifold. Let $\varphi : \Sigma \to M$ be a minimal immersion that does not cross itself and where $\Sigma$ has no boundary. Every point $x_1 \in \Sigma$ has a neighborhood $U_1 \subset \Sigma$ with the following property. Whenever $x_2 \in \Sigma$ is such that $\varphi (x_1) = \varphi (x_2)$
there is a neighborhood $U_2 \subset \Sigma$ with $x_2 \in \Sigma$ and a diffeomorphism $\psi : U_1 \to U_2$
so that
\[
\psi (x_1) = x_2 \quad \text{ and } \quad \varphi \circ \psi = \varphi.
\]
\end {lemma}


\begin {lemma} \label{lem:generalquotient} Let $(M, g)$ be a Riemannian manifold. Let $\varphi : \Sigma \to M$ be a connected minimal immersion that does not cross itself where $\Sigma$ has no boundary. We say that two points $x_1, x_2 \in \Sigma$ are equivalent and write 
\[
x_1 \sim x_2 \quad \text{ if } \quad \varphi(x_{1}) = \varphi(x_{2}).
\]
The topological quotient $\tilde \Sigma = \Sigma /{\sim}$ is a smooth manifold. The quotient map 
\[
\pi : \Sigma \to \tilde \Sigma \quad \text{ given by } \quad x \mapsto [x]_{\sim}
\]
is a covering. There is a unique immersion $\tilde \varphi : \tilde \Sigma \to M$ such that the diagram
\[
\xymatrix{
\Sigma \ar[d]_{\pi} \ar[dr]^{\varphi}\\
\tilde \Sigma \ar[r]_{\tilde\varphi} & M
}
\]
commutes.
\end {lemma} 

\begin{remark} \label{rem:willbeinjective}

Let $\varphi : \Sigma\to M$ be a connected two-sided minimal immersion that does not cross itself. Let $\nu : \Sigma \to T M$ be a unit normal field. A variant of the preceding lemma where we only identify points $x_1, x_2 \in \Sigma$ with 
\[
\varphi(x_1) = \varphi(x_2) \quad \text{ and } \quad \nu(x_1) = \nu(x_2)
\]
allows us to factor through to a two-sided  minimal immersion $\tilde\varphi:\tilde\Sigma \to M$ that is either injective or two-to-one by a \emph{side-preserving} covering $\pi : \Sigma \to \tilde \Sigma$. A useful example to bear in mind in this context is the minimal immersion $\mathbb{S}^2 \to \mathbb{R P}^3$ obtained from following the antipodal map $\mathbb{S}^2 \to \mathbb{R P}^2$ by the equatorial embedding $\mathbb{R P}^2 \to \mathbb{R P}^3$.
\end{remark}

\begin {remark} A cover of a stable minimal immersion\footnote{Recall that in this paper a stable minimal immersion is by definition two-sided.} is still stable; see \cite{Fischer-Colbrie-Schoen}. The converse of this statement is not true in general; see \cite[p.\ 491]{Meeks-Rosenberg:2006}. 
\end {remark}

The next result is a special case of \cite[Lemma A.1 (2)]{Meeks-Rosenberg:2006}. We include an outline of the proof for convenient reference. 

\begin{lemma}[\cite{Meeks-Rosenberg:2006}] \label{lem:stability-factors}
Let $(M, g)$ be a Riemannian $3$-manifold. Let 
$
\varphi: \Sigma \to M
$
be a complete stable minimal immersion such that $\Sigma$ with the induced metric is conformal to the plane. Let  
$
\pi : \Sigma \to \tilde \Sigma
$ 
be a side-preserving covering of surfaces where $\tilde \Sigma$ is non-compact. The map   
$
\tilde \varphi: \tilde \Sigma \to M
$
that makes the diagram 
\[
\xymatrix{
\Sigma \ar[d]_{\pi} \ar[dr]^{\varphi}\\
\tilde \Sigma \ar[r]_{\tilde\varphi} & M
}
\]
commute is a complete stable minimal immersion. 

\begin{proof} It suffices to consider the case where \[\Sigma = \R \times \R \quad \text{ and } \quad \tilde \Sigma = \R \times \R / \mathbb{Z}\] and where \[\pi : \Sigma \to \tilde \Sigma \quad \text{ is given by } \quad (x, y) \mapsto ( x, [y]).\] Let $T > 0$. Assume that the domain $(-T , T) \times \mathbb{R}/\mathbb{Z}$ is unstable for the immersion \[\tilde \varphi : \R \times \R / \mathbb{Z} \to M.\] It follows that there is $\tilde u \in C^\infty_{c}( (-T, T) \times \mathbb{R}/\mathbb{Z})$ and $\delta > 0$ such that \[\delta +  \int_{(-T,T) \times \mathbb{R}/\mathbb{Z}}  |\tilde \nabla \tilde u|^{2}  \leq   \int_{(-T,T) \times \mathbb{R}/\mathbb{Z}}  (|\tilde h|^{2}+\Ric(\tilde \nu, \tilde \nu))\tilde u^{2}.\] Let $\chi \in C^{\infty}((-3,0))$ be such that $\chi(x) = 1$ for $x > -1$ and $\chi(x) = 0$ for $x < -2$. Given $n \geq 1$, consider the cut-off function $\chi_n \in C^{\infty}_{c}((-3,n+3))$ given by \[ \chi_n(x) =  \begin{cases} \chi(x) & x \in (-3,0)\\ 1 & x \in [0,n]\\ \chi(-x+n) & x \in (n,n+3). \end{cases}\] Set $ u_n = (\tilde u \circ \pi)\chi_n \in C^\infty_{c}(\mathbb{R}^{2})$. Using the stability of the immersion \[\varphi: \R \times \R \to M\] and equivariance of $\varphi$ in the second component, we obtain that \begin{align*} 0  & \geq \int_{(-T, T) \times (-3,n+3)} (|h|^{2}+\Ric (\nu,\nu)) u_n^{2} - |\nabla  u_n|^{2} \\ & = n\delta + \int_{(-T , T) \times ((-3,0)\cup(n,n+3))} (|h|^{2}+\Ric(\nu,\nu)) u_n^{2} - |\nabla  u_n|^{2} \\ & = n\delta - c \end{align*} where the constant $c$ is independent of $n$. Taking $n$ sufficiently large, we obtain a contradiction.
\end{proof}
\end{lemma}


\section {Barriers for the functional $\Omega \mapsto \area (\partial \Omega) - H \vol(\Omega)$}

\begin {lemma} \label{lem:effectivebarrier}

Let $g$ be a Riemannian metric on $B_2 (0) \times (-2, 2) \subset \R^{m+1}$ and let $u \in C^\infty(B_2 (0))$ have the following properties: 
\begin {enumerate} [(i)]
\item $-2 \leq u (x) \leq 2$ for all $x \in B_2(0)$. 
\item $u(x) > 0$ for all $x \in B_1(0)$.
\item $u(x) \leq 0$ for all $x \in B_2(0)$ with $|x| > 1$.  
\end {enumerate} 
Let $\varepsilon \in (0, 1)$. The region 
\[
D_\varepsilon = \{ (x, z) : x \in B_1(0) \text{ and } 0 < z< \varepsilon  u(x) \}
\]
is foliated by level sets of the function 
\[
v : B_1(0) \times (-2, 2) \to \R \quad \text{ given by } \quad (x, z) \mapsto \frac{z}{u(x)}. 
\]
Let $(x_0, z_0) \in D_\varepsilon$. The vector field
\[
X = \frac{\nabla v}{ |\nabla v|}
\]
at the point $(x_0, z_0)$ is equal to the upward pointing unit normal of the graph 
\[
x \mapsto \left(x, \frac{z_0 }{u(x_0)} u(x) \right)
\]
and its divergence at $(x_0, z_0)$ is equal to the mean curvature of this graph computed with respect to the upward pointing unit normal at $(x_0, z_0)$. As $\varepsilon \searrow 0$, the mean curvatures of these graphs approach the mean curvature of the disk $B_1(0) \times \{0\}$ where we identify points with the same first coordinate. 
\end {lemma}


\section {Variation formulae for area and relative volume} \label{sec:variationformulae}

In this section, we recall the first and second variation formulae for the area and (relative) volume of immersions. We refer the reader to e.g. \cite{Barbosa-doCarmo-Eschenburg:1988, Bray:2001} for derivations.

Let $(M, g)$ be a Riemannian manifold. We consider a two-sided immersion $\varphi : \Sigma \to M$ with unit normal $\nu : \Sigma \to T M$ and mean curvature $H \in C^\infty(\Sigma)$. (We always take the mean curvature to be the tangential divergence of the designated unit normal. The mean curvature vector field is thus given by $- H \nu$.) We also assume that neither $M$ nor $\Sigma$ have boundary. 

Let $U \in C^\infty(\Sigma \times (- \varepsilon, \varepsilon))$ be compactly supported in $\Sigma$ and such that $U(x, 0) = 0$ for all $x \in \Sigma$. Shrinking $\varepsilon > 0$, if necessary, we obtain a variation $\{\varphi_t : \Sigma \to M\}_{t \in (- \varepsilon, \varepsilon)}$ of $\varphi : \Sigma \to M$ through immersions  
\[
\varphi_t : \Sigma \to M \quad \text{ given by } \quad x \mapsto \exp U(x, t) \nu(x).
\]
Except for reparametrizations, every variation of $\varphi : \Sigma \to M$ arises in this way. We have that 
\begin {align*}
\frac{d}{dt} \Big|_{t = 0} \text{area of } \varphi_t &= \int_\Sigma H \,  \dot U (\, \cdot \,, 0) \\
\frac{d^2}{dt^2} \Big|_{t = 0} \text{area of } \varphi_t &= \int_\Sigma H^2\,  \dot U (\, \cdot \,, 0)^2 + H \,  \ddot U (\, \cdot \,, 0)+ |\nabla  (\dot U (\, \cdot \,, 0))|^2 - (|h|^2 + \Ric (\nu, \nu)) \dot U (\, \cdot \,, 0)^2. 
\end{align*}
Observe the abuse of notation here: the area of $\varphi_t$ may be infinite. Instead, we should consider the measure (with respect to the induced metric) of the spatial support of $U$ or that of any compact subset of $\Sigma$ containing it. Assume now that $M$ is oriented. We define the relative volume of $\varphi_t$ by integrating the pull-back of the volume form of $(M, g)$ by the map 
\[
(x, t) \mapsto \varphi_t(x)
\]
across $\Sigma \times [0, t]$ if $t \geq 0$ and across $\Sigma \times [t, 0]$ if $t < 0$. It follows that
\begin {align*}
\frac{d}{dt} \Big|_{t = 0} \text{relative volume of } \varphi_t &= \int_\Sigma \dot U (\, \cdot \,, 0) \\
\frac{d^2}{dt^2} \Big|_{t = 0} \text{relative volume of } \varphi_t &= \int_\Sigma \ddot U (\, \cdot \,, 0)  + H \,  \dot U (\, \cdot \,, 0)^2. 
\end{align*}
Here, dots indicate derivatives with respect to the variation parameter, $h$ is the second fundamental form of $\varphi : \Sigma \to M$, $\Ric$ is the ambient Ricci curvature, and integration, gradient, and lengths are taken with respect to the induced metric $\varphi^* g$ on $\Sigma$. 

The special case where $U (x, t) = t u (x)$ for all $(x, t) \in \Sigma \times (- \varepsilon, \varepsilon)$ where $u \in C^\infty_c(\Sigma)$ is particularly important. Note that $\dot U (\, \cdot \, , 0) = u$ in this case.


\section {Area-minimizing boundaries} \label{sec:am}

Let $(M, g)$ be an asymptotically flat Riemannian $3$-manifold. Extend $M$ inwards across each of its minimal boundary components by thin collar neighborhoods to a new manifold $\hat M$ without boundary. Denote the union of these finitely many collar neighborhoods by $C$. We consider the collection $\mathcal{F}$ of all properly embedded $3$-dimensional submanifolds with boundary $\Omega \subset \hat M$  with $C \subset \Omega$. A surface $\Sigma \subset M$ \emph{bounds} in $M$ \emph{relative to} $\partial M$ if it arises as the boundary of such a smooth region. Note that $\partial M$ bounds in this sense. 

We say that the boundary of $\Omega \in \mathcal{F}$ is \emph{area-minimizing} if for all $\rho > 1$ and $\tilde \Omega \in \mathcal{F}$ with $ \Omega \setminus B_\rho = \tilde \Omega \setminus B_\rho$ we have that
\[
\area (B_{2\rho} \cap \partial \Omega) \leq \area(B_{2\rho} \cap \partial  \tilde \Omega).\footnote{We could also work with the larger class of all $3$-dimensional submanifolds with locally finite boundary area. However, by standard geometric measure theory, every \emph{such} submanifold with area-minimizing boundary \emph{is} properly embedded.}
\]
The components of the boundary $\Sigma \subset M$ of such $\Omega \in \mathcal{F}$ are stable minimal surfaces in $(M, g)$. It follows from the arguments in Section 6 of \cite{isostructure} that $\Sigma$ has at most one unbounded component $\Sigma_0$. More precisely, if we consider the homothetic blow-downs of $\Omega$ in the chart at infinity 
\[
M \setminus U \cong \{x \in \R^3 : |x| > 1\}
\]
by a sequence $\lambda_i \to \infty$, then we can pass to a geometric subsequential limit in $\R^3 \setminus \{0\}$. This limit is either a half-space through the origin, or empty in the case where $\Omega$ is bounded. 


\section {A particular conformal change of metric} \label{sec:cf}

Deformations of the metric tensor such as the ones considered here have been studied in great generality by P. Ehrlich in \cite{Ehrlich:1976}; see also the paper by G. Liu \cite{Liu:2013}.

Let $f \in C^\infty(\R)$ be a non-positive function with support in the interval $[0, 3]$ such that 
\[
f(s) = - \exp (18 / (s-3))
\]
when $s \in (1, 3)$. This definition is made so that 
\[
0 <  f'(s) \quad \text{ and } \quad s f'' (s) + 3 f' (s) < 0 
\]
for all $s \in (1, 3)$. 

Let $(M, g)$ be a homogeneously regular Riemannian $3$-manifold. Choose $0 < r_0 < \text{inj} (M, g)/4$ so that 
\[
\Delta_{g} \dist_g(\, \cdot\, , p)^2 \leq 8 \quad \text{ on } \quad \{ x \in M : \dist_g(x, p) \leq 3 r_0\}
\] 
for all $p \in M$. Here, $\Delta_g$ is the non-positive Laplace-Beltrami operator with respect to $g$. Fix $p \in M$ and $0 < r \leq r_0$. Consider the function $v : M \to \R$ given by
\[
x \mapsto r^4 f ( \dist_g(x, p)/r).
\]
Note that $v$ is smooth, non-positive, and supported in $\{x \in M : \dist_g(x, p) \leq 3 r\}$. Moreover, 
\[
v< 0 \qquad  \text{ and }  \qquad \Delta_g v < 0
\]
on $\{ x \in M : r < \dist_g(x, p) < 3 r \}$.

For $\epsilon > 0$ sufficiently small, a smooth family of conformal metrics $\{g(t)\}_{t \in (0, \epsilon)}$ with the properties needed in the proof of Theorem \ref{thm:Schoenrigidity} is given by
\[
g(t) = (1 + t v)^4 g. 
\]
In fact, when $u \in C^\infty(M)$ is positive and $\tilde g = u^4 g$ is a conformal metric, we have that
\[
u^5 R_{\tilde g} = u R_g - 8  \Delta_g u
\]
and 
\[
u^3 H_{\tilde g} = u H_g + 4 (\text{d} u)(\nu_g )
\]
along $\Sigma$ where $\nu_g$ is the unit normal with respect to $g$. 


\section{Existence of isoperimetric regions of all volumes} \label{sec:appendixisoallvolume}

S. Brendle and the second-named author observed \cite{BrendleChodosh} that the monotonicity of the Hawking mass along G. Huisken and T. Ilmanen's weak inverse mean curvature flow \cite{Huisken-Ilmanen:2001} can be combined with the co-area formula to give an explicit lower bound for the volume swept out under inverse mean curvature flow of a surface. This insight was subsequently used by the second-named author to study the large isoperimetric regions of asymptotically hyperbolic manifolds \cite{Chodosh:largeIso}. In a recent preprint, Y.\ Shi \cite{Shi:isoIMCF}  observed that closely related arguments can be used to construct regions whose isoperimetric ratio is better than Euclidean in non-flat asymptotically flat manifolds that have non-negative scalar curvature. Here we note that Y. Shi's observation implies the existence of isoperimetric regions of \emph{all} volumes in asymptotically flat $3$-manifolds with non-negative scalar curvature. This answers a question of G. Huisken.

\begin{prop}
Let $(M,g)$ be an asymptotically flat Riemannian $3$-manifold with horizon boundary and non-negative scalar curvature. Then $(M, g)$ admits isoperimetric region for every volume, i.e., for every $V > 0$ there is a smooth bounded region $\Omega_V \subset M$ of volume $V$ that contains the horizon such that 
\begin{align} 
\area (\partial \Omega_V) = \inf \{ \area (\partial \Omega) : \Omega \subset M \text{ smooth open region containing the horizon, volume } V\}.
\end{align}
\end{prop}
\begin{proof}
The first part of the argument is as in \cite{Shi:isoIMCF}. Let $r>0$. We claim that there are bounded Borel sets $\Omega$ with finite perimeter $\Omega$ that lie arbitrarily far out in the asymptotically flat region of $(M, g)$ such that 
\[
\mathcal{H}^2_{g}(\partial^{*}\Omega) = 4\pi r^{2} \qquad\text{and} \qquad \mathcal{L}^3_{g}(\Omega) > \frac{4}{3}\pi r^{3}. 
\]
To see this, fix a point $p \in M$ that lies far out in the asymptotic region of $(M,g)$ and so that $(M,g)$ is not flat at $p$. Let $\Omega_{\tau} = \{u<\tau\}$ denote the region swept out by the weak inverse mean curvature flow ``starting at the point $p$'' as constructed in \cite[Lemma 8.1]{Huisken-Ilmanen:2001}. We may assume (by \cite[Lemma 1.6]{Huisken-Ilmanen:2001}) that $\mathcal{H}^2_{g}(\partial^{*}\Omega_{\tau}) = 4\pi e^{\tau}$. Because the scalar curvature of $(M,g)$ is non-negative and $g$ is non-flat at $p$, the Hawking mass of $\partial^{*}\Omega_{\tau}$ is strictly positive for all $\tau > 0$. Thus, the argument in \cite[Proposition 3]{BrendleChodosh} or in \cite{Shi:isoIMCF} shows that
\[
\mathcal{L}^3_{g}(\Omega_{\tau}) > 2\pi \int_{-\infty}^{\tau} e^{\frac{3t}{2}} dt = \frac{4\pi}{3} e^{\frac{3\tau}{2}}
\]
for all $\tau > 0$. Choosing $\tau=2\log r$ we obtain the desired region. 

Using that $(M, g)$ has horizon boundary, we see that its isoperimetric profile is strictly increasing as in the proof of \cite[Lemma 3.3]{Chodosh:largeIso}. The result now follows from \cite[Proposition 4.2]{isostructure} or \cite[Theorem 2]{Nardulli:existence}. 
\end{proof}

We also mention the existence results of A. Mondino and S. Nardulli \cite{Mondino-Nardulli:2012} for isoperimetric regions of all volumes in complete and non-compact Riemannian manifolds that satisfy a lower bound on the Ricci curvature and are locally asymptotic to model geometries. Their results are based on the comprehensive analysis of S. Nardulli \cite[Theorem 2] {Nardulli:existence}.


\bibliography{bib} 
\bibliographystyle{amsplain}
\end{document}